\documentclass[a4paper]{amsart}
\usepackage{amssymb, enumerate}
\usepackage{stmaryrd}
\usepackage{tikz}
\usepackage{bbm}
\usepackage[all]{xy}

\usepackage{hyperref, aliascnt}
\usepackage{mathtools}
\usepackage{comment}
\usepackage{mathrsfs}
\def\today{\number\day\space\ifcase\month\or   January\or February\or
   March\or April\or May\or June\or   July\or August\or September\or
   October\or November\or December\fi\   \number\year}

\theoremstyle{definition}
\newtheorem{lma}{Lemma}[section]

\newaliascnt{thmCt}{lma}
\newtheorem{thm}[thmCt]{Theorem}
\aliascntresetthe{thmCt}

\newaliascnt{corCt}{lma}
\newtheorem{cor}[corCt]{Corollary}
\aliascntresetthe{corCt}

\newaliascnt{propCt}{lma}
\newtheorem{prop}[propCt]{Proposition}
\aliascntresetthe{propCt}

\newtheorem*{thm*}{Theorem}
\newtheorem*{qst*}{Question}

\newtheorem*{cnj*}{Conjecture}
\newtheorem*{cor*}{Corollary}
\newtheorem*{prop*}{Proposition}

\newcounter{theoremintro}
\newtheorem{thmintro}[theoremintro]{Theorem}
\newtheorem{corintro}[theoremintro]{Corollary}

\newtheorem{cnjintro}[theoremintro]{Conjecture}

\newaliascnt{pgrCt}{lma}

\aliascntresetthe{pgrCt}

\newaliascnt{dfCt}{lma}
\newtheorem{df}[dfCt]{Definition}
\aliascntresetthe{dfCt}

\newaliascnt{remCt}{lma}
\newtheorem{rem}[remCt]{Remark}
\aliascntresetthe{remCt}

\newaliascnt{remsCt}{lma}

\aliascntresetthe{remsCt}

\newaliascnt{egCt}{lma}
\newtheorem{eg}[egCt]{Example}
\aliascntresetthe{egCt}

\newaliascnt{egsCt}{lma}

\aliascntresetthe{egsCt}

\newaliascnt{qstCt}{lma}
\newtheorem{qst}[qstCt]{Question}
\aliascntresetthe{qstCt}

\newaliascnt{pbmCt}{lma}

\aliascntresetthe{pbmCt}

\newaliascnt{notaCt}{lma}

\aliascntresetthe{notaCt}

\newaliascnt{cnjCt}{lma}

\aliascntresetthe{cnjCt}

\newcommand{\beq}{\begin{equation}}
\newcommand{\eeq}{\end{equation}}
\newcommand{\beqa}{\begin{eqnarray*}}
\newcommand{\eeqa}{\end{eqnarray*}}
\newcommand{\bal}{\begin{align*}}
\newcommand{\eal}{\end{align*}}
\newcommand{\bi}{\begin{itemize}}
\newcommand{\ei}{\end{itemize}}
\newcommand{\be}{\begin{enumerate}}
\newcommand{\ee}{\end{enumerate}}

\newcommand{\N}{{\mathbb{N}}}

\newcommand{\G}{{\mathcal{G}}}

\newcommand{\B}{{\mathcal{B}}}

\newcommand{\id}{{\mathrm{id}}}

\newcommand{\supp}{{\mathrm{supp}}}

\newcommand{\Gn}[1]{\mathcal{G}^{(#1)}}

\newcommand{\I}{\infty}

\newcommand{\Set}[1]{\{ #1 \}}
\newcommand{\acts}{\curvearrowright}

\date{\today}
\title[Generalisations of Thompson's group from purely infinite groupoids]{Generalisations of Thompson's group V arising from purely infinite groupoids}

\thanks{
The first named author was supported by a starting grant of 
the Swedish Research Council.
The second named author has received funding from the European Research Council (ERC) under the European Union's Horizon 2020 research and innovation programme (grant agreement No. 817597). Much of the work of this project was undergone during a visit of the second named author to the WWU M\"unster, where the second author was partially supported by the Deutsche Forschungsgemeinschaft (DFG, German Research Foundation) under Germany's Excellence Strategy EXC 2044 –390685587, Mathematics M\"unster: Dynamics–Geometry–Structure, and through SFB 1442, and by the European Research Council through ERC Advanced Grant 834267 - AMAREC}

\author{Eusebio Gardella}
\address[Eusebio Gardella]
{Department of Mathematical Sciences, Chalmers University of
Technology and University of Gothenburg, Gothenburg SE-412 96, Sweden.}
\email{gardella@chalmers.se}
\urladdr{www.math.chalmers.se/~gardella}

\author{Owen Tanner}
\address[Owen Tanner]{School of Mathematics and Statistics,
University of Glasgow,
\linebreak University Place, Glasgow G12 8QQ, United Kingdom}
\email[]{o.tanner.1@research.gla.ac.uk}
\urladdr{https://owentanner1997.wordpress.com/}

\begin{document}

\begin{abstract}
We study a class of generalisations of Thompson's group $V$ arising naturally as topological full groups of purely infinite, minimal groupoids. In the process, we show that the derived subgroup of such a group has no proper characters, and that it is  2-generated whenever it is finitely generated. We characterise the class of purely infinite, minimal groupoids through a number of group-theoretic conditions on their full groups including vigor, the existence of suitable embeddings of $V$, and compressibility. We moreover give a complete abstract description of those groups that arise as either topological full groups or derived subgroups of purely infinite, minimal groupoids. 
As an application, we obtain several new facts about 
the Brin-Higman-Thompson groups, including C*-simplicity and a description of all
their proper characters.
\end{abstract}

\maketitle

\renewcommand*{\thetheoremintro}{\Alph{theoremintro}}

\section{Introduction}

In his famous unpublished notes from 1967 \cite{thompson},
Thompson introduced three groups which are now most
commonly denoted $F\subseteq T\subseteq V$. 
Although Thompson's motivation for considering these groups came from algebraic
logic, they received a great deal of attention in the context of 
the von Neumann problem (which asked whether every 
nonamenable group must contain the free group $\mathbb{F}_2$). It is known that $T$ and $V$ contain a free group, so that they cannot possibly give a 
negative answer to the problem. On the other hand, $F$ is known not to contain free groups, but it remains open whether it is amenable or not. Beyond the von 
Neumann problem, these groups are quite remarkable in that they exhibit a number of unusual properties. In the setting of group
theory, the role of the group $V$ cannot be overstated. 
For example, $V$ was the first example of a simple, infinite and finitely presented group, 
and as such it remains a fundamental object in the study of combinatorial group theory. For a discussion and introductory notes on Thompson's groups, we refer the reader to \cite{cannon2011thompson, cannon1996introductory}. 

Since the introduction of $V$, group theorists have created vast families of groups generalising it, which for the purposes of this introduction we call \emph{Thompson-like groups}. Examples of such groups have been studied by Higman \cite{higman}, Brin \cite{brin}, Cleary \cite{cleary1995groups} and Stein \cite{stein1992groups}. These groups are very diverse, and one aspect that unifies them is that they have simple derived subgroups with certain finiteness properties. 

In this work, we take a groupoid perspective on 
Thompson-like groups, and study them from the point of
view of topological full groups and the associated
derived subgroups. More explicitly,
beginning with an ample, essentially principal groupoid $\mathcal{G}$, its topological full group $\mathsf{F}(\mathcal{G})$ is defined as the group of homeomorphisms of the unit space that are piecewise formed from finitely many compact open bisections. In essence, one builds global symmetries from local ones; see \autoref{groupoid tfg} for the 
precise definition. We now know much about the subgroup structure of these groups; for example, the derived subgroup  
$\mathsf{D}(\mathcal{G})=\big[\mathsf{F}(\G),\mathsf{F}(\G)\big]$ is often simple, and often agrees with the so-called alternating group $\mathsf{A}(\mathcal{G})$, which is the subgroup of $\mathsf{F}(\G)$ generated by certain elements of order three.

Since its conception, the framework of topological full groups has served as a bridge between Thompson-like groups and the groupoid models of purely infinite C*-algebras, a line of research that was pioneered by Matui \cite{Mat_homology_2012}. 
In fact, promising observations by Nekrashevych \cite{nekrashevych2004cuntz} predate, and even motivate, the definition of topological full groups of \'etale groupoids by Matui. Because the underlying topological groupoids are often much more accessible to study than the Thompson-like groups themselves, this new dynamical perspective has led to progress in our understanding of the homology \cite{Li_ample_2022}, finite presentation \cite{li2021left}, and subgroup structure of Thompson-like groups \cite{bon2018rigidity}.

We summarise the correspondence between Thompson-like groups and purely infinite groupoids in the table below, where $\mathcal{E}_k$ is the full shift on $k$ generators, $\mathcal{R}_r$ is the full equivalence relation on $r$ generators, $\mathcal{O}_k$ is the Cuntz algebra on $k$ generators, and $\mathcal{Q}^\lambda$ is the class of Kirchberg algebras considered in \cite{li2015new, carey2011families}: 

\begin{figure}[h!]
    \centering
  \scalebox{1.0}{ \begin{tabular}{|c|c|c|}
    \hline \textbf{C$^*$-algebra} &
    \textbf{Groupoid}   &  \textbf{Thompson-like group}   \  \\ \hline
    
    $\mathcal{O}_2$ & $\mathcal{E}_2$    & $V$ \ \\
    
    $M_r(\mathcal{O}_k)$& $\mathcal{R}_r \times \mathcal{E}_k $ & $ V_{k,r}$  \ \\
     $M_r(\bigotimes_{j=1}^n\mathcal{O}_k)$& $\mathcal{R}_r \times \mathcal{E}_k^n $ & $n V_{k,r}$  \ \\
     $\mathcal{Q}^\lambda$ & Particular partial actions & Irrational slope Thompson groups  \ \\ \hline
     \end{tabular}
     }
    \caption{Realisations of Thompson-like groups as topological full groups}
    \label{table 1}
\end{figure}

The dynamical realisation of the Brin-Higman-Thompson groups is described in \cite{Mat_etale_2016}, while the realisation of irrational slope Thompson groups as topological full groups is the subject of the second author's PhD thesis \cite{mythesis}. 

In light of this correspondence between groupoids and Thompson-like groups, Matui proposed in \cite{Mat_topological_2015}
to regard the topological full groups of other generalisations of the full shift on two generators as generalisations of $V$, and in particular
study them as Thompson-like groups. In turn, analysing the topological full groups led us to the discovery of interesting new groups, many of which exhibit properties enjoyed by Thompson's group $V$. This vein of research is summarised in the table below: 

\begin{figure}[h!]
    \centering
  \scalebox{0.90}{ \begin{tabular}{|c|c|c|}
    \hline \textbf{C$^*$-algebra} &
    \textbf{Groupoid $\mathcal{G}$}   & Is $\mathsf{A}(\mathcal{G})$ simple and \  \\  & & finitely presented? \ \\ \hline
  Cuntz-Krieger algebras & SFT Groupoids & Yes  \ \\
   Tensors of Cuntz-Krieger algebras & Products of SFT groupoids & Yes \ \\
   Graph algebras & Graph groupoids & Yes  \ \\
  Katsura-Exel-Pardo algebras & Katsura-Exel-Pardo groupoids & Yes \ \\

\hline
     \end{tabular}} 
    \caption{New generalisations of Thompson's group $V$ from topological full groups}
    \label{Table 2}
\end{figure}

Very general structural properties have been studied for these examples in \cite{Mat_etale_2016}, and significant progress toward homological properties has been made in \cite{Li_ample_2022}. In particular, topological full groups inside this class give rise to many new examples of finitely presented, simple, infinite groups \cite{Li_ample_2022, li2021left, MatMat_topological_2014, Mat_topological_2015, NylOrt_topological_2019, NylOrt_matuis_2021, NylOrt_katsura_2021}.

C*-algebraists have a deep source of such groupoids which have been developed primarily as a way to construct simple, purely infinite, nuclear C*-algebras -- also known as Kirchberg algebras; see \cite{Spi_models_2007}. The relevance of these examples stems from the fact that such C*-algebras play a crucial role in the classification theory of simple, nuclear C*-algebras pioneered by Elliott; see \cite{phillips2000classification}. Each of the underlying groupoids is minimal, topological principal and purely infinite in the sense of Matui \cite{Mat_topological_2015}.

In this paper, we study the topological full groups of minimal, purely infinite groupoids. We show that this class of groupoids can be characterised in a number of ways intrinsic to the topological full groups. We summarise these characterisations 
in the following result:

\begin{thmintro}\label{thmintro:equiv}
    Let $\G$ be an essentially principal, ample groupoid. Then the following are equivalent:
    \begin{enumerate}
        \item $\G$ is purely infinite and minimal.
        \item $\mathsf{A}(\G)$ is a vigorous subgroup of $\mathrm{Homeo}(\G^{(0)})$ in the sense of Bleak-Elliott-Hyde \cite{BleEllHyd_sufficient_2020}. 
        \item For every $x_0 \in \G^{(0)}$, the subgroup $$\mathsf{A}(\G)_{x_0}=\{ g \in \mathsf{A}(\G) \colon \text{there exists a neighbourhood } Y \text{ of } x_0 \text{ such that } g|_Y=1 \}$$
        acts compressibly on $\G^{(0)}\setminus \{x_0\}$ in the sense of  Dudko-Medynets \cite{DudMed_finite_2014}.  
         \end{enumerate}
         Moreover, if any of the above holds, then
for every nontrivial compact open subset $X \subsetneq \G^{(0)}$, there is an embedding $$\phi_X\colon V \rightarrow \mathsf{A}(\G)$$ such that $X \subseteq \supp (\phi_X(V))$.
\end{thmintro}

\autoref{thm:equiv} completely answers Question~6.1 of \cite{openproblemsworkshop}, which asks to determine when a topological full group is vigorous. The proof of the above theorem is completed at the end of \autoref{s3}.

The last condition in the theorem above shows some form of universality enjoyed by $V$ in the context
of purely infinite, minimal groupoids and their topological full groups, which is reminiscent of similar properties enjoyed by the 
Cuntz algebra $\mathcal{O}_2$ among purely infinite, simple C*-algebras.

The fact that conditions (2) and (3) characterise minimality and 
pure infiniteness of $\G$ is surprising, since vigor and compressibility were introduced in a completely different context without having groupoids in mind, and with the goal of further developing the understanding of Thompson's group $V$ and its generalisations. We discuss both vigor and compressibility separately, since our results have 
interesting applications in both settings.

Vigor was introduced in \cite{BleEllHyd_sufficient_2020} in order to determine when a simple, finitely generated group is two-generated. More specifically, said condition was introduced in order to understand the following well-established conjecture:

\begin{cnjintro}\label{fpresent2present} Every simple, finitely presented group is two-generated.  \end{cnjintro} 

Our first corollary confirms this conjecture for a broad class of derived subgroups in topological full groups; see \autoref{cor:twoGeneration}.

\begin{corintro} Let $\G$ be a minimal, expansive, purely infinite, essentially principal Cantor groupoid. Then $\mathsf{A}(\G)$ is two-generated. \end{corintro}

On the other hand, compressibility was introduced in \cite{DudMed_finite_2014} in order to understand the representation theory and dynamical properties of the Higman-Thompson groups $V_{k,r}$. This definition is an analogue of compressibility in the realm of measurable dynamics, and the introduction of this property follows the recent trend of ideas in measurable dynamics being imported directly to topological dynamics. This gives us two main facts about $\mathsf{F}(\G)$; see \autoref{D has no proper characters} and \autoref{finite factor}:

\begin{corintro} \label{nopropercharacters} Let $\G$ be a minimal, purely infinite, essentially principal Cantor groupoid. Then $\mathsf{A}(\G)$ has no proper characters, and there are no nontrivial finite factor representations of $\mathsf{A}(\G)$. 
\end{corintro}

One way to interpret this is that $\mathsf{A}(\G)$ is highly nonlinear, being an infinite simple group. This allows us to reduce the question of understanding the representations of $\mathsf{F}(\G)$ to understanding the abelianisation $\mathsf{F}(\G)_{\mathrm{ab}}=\mathsf{F}(\G)/\mathsf{D}(\G)$. This abelianisation can be described very concretely through Matui's AH conjecture, which has recently been verified by Xin Li \cite{Li_ample_2022} for the class of groupoids considered in this paper. We give a wealth of concrete computations: for example, we are able to describe concretely all characters of the Brin-Higman-Thompson groups $nV_{k,r}$; see 
\autoref{brin}. The property of compressibility also gives us dynamical information for actions of Thompson-like groups;
see \autoref{essentially free automatic}:

\begin{corintro} Let $\G$ be an essentially principal, purely infinite, minimal, Cantor groupoid such that $H_1(\mathsf{F}(\G))$ is finite (this is for example
the case if $H_0(\G)$ and $H_1(\G)$ are finite). Then every faithful ergodic measure-preserving action of $\mathsf{F}(\G)$ is essentially free.  \end{corintro}

The fact that $V$ embeds into $\mathsf{F}(\G)$ whenever $\G$ is minimal and purely infinite is in-and-of itself useful for understanding key properties of $\mathsf{F}(\G)$. As a consequence, many free products are in $\mathsf{F}(\G)$, including the product $\mathbb{F}_2\times \mathbb{F}_2$ of the free group on two generators with itself. Also, it allows us to conclude that the generalised word problem is not solvable for $\mathsf{F}(\G)$. Crucially, this means that many of these topological full groups, for example many groupoids in the Garside category framework introduced by Li in \cite{li2021left} (a broad class whose topological full groups include all Brin-Higman-Thompson groups and Stein groups), have solvable word problem but not solvable generalised word problem. (For the precise definition of the (generalised) word problem, see \autoref{wordproblem}.) In other words, these groups sit on an important boundary for computability.

The methods we develop in this work allow us to 
give a complete abstract characterisation of the 
full and derived subgroups of minimal purely infinite
groupoids:

\begin{thmintro}\label{thmintro:Realization}
Let $\mathcal{C}$ denote the Cantor space.
\begin{enumerate}
    \item For a subgroup $F\leq \mathrm{Homeo}(\mathcal{C})$, the following are equivalent:
\begin{enumerate} 
\item[(F.1)] There exists a minimal, purely infinite, essentially principal, Cantor \'etale groupoid $\G_F$ with $\mathsf{F}(\G_F)\cong F$, and
\item[(F.2)] $F$ is vigorous and locally closed.
\end{enumerate}
\item For a subgroup $A\leq \mathrm{Homeo}(\mathcal{C})$, the following are equivalent:
\begin{enumerate}     \item[(A.1)] There exists a minimal, purely infinite, essentially principal, Cantor \'etale groupoid $\G_A$ with $\mathsf{A}(\G_A)\cong A$, and
\item[(A.2)] $A$ is vigorous and simple.
\end{enumerate}
\end{enumerate}

Moreover, the groupoids $\G_F$ and $\G_A$ as in (F.1) and (A.1) 
above are unique
up to groupoid isomorphism.
\end{thmintro}

As an application, 
we show in \autoref{cor:C*simple} that if 
$A\leq F\leq\mathrm{Homeo}(\mathcal{C})$ are 
nested groups such that $A$ is vigorous and simple,
and $F$ is locally closed, then any intermediate
group $A\leq H\leq F$ is C*-simple. In particular, any
vigorous, simple subgroup of $\mathrm{Homeo}(\mathcal{C})$ is C*-simple, and the same applies to
any vigorous subgroup which is locally closed. 
Specifically for Thompson's group $V$, this had been shown by Le Boudec and Matte-Bon in \cite{le2016subgroup}. 


Our final section concerns applications to the examples described in the tables above. To illustrate, we enumerate a number of new facts about the Brin-Higman-Thompson groups $nV_{k,r}$:

\begin{itemize}
    \item[(i)] $nV_{k,r}$ is C*-simple for all $n, r\geq 1$ and $k\geq 2$.
    \item[(ii)] The characters and finite factor representations of
    $nV_{k,r}$ can be explicitly described; see \autoref{brin} for details.
    \item[(iii)] For any $H$ with $\mathsf{D}(n V_{k,r}) \leq H \leq nV_{k,r}$, any faithful, ergodic, measure-preserving action of $H$ is automatically essentially free. 
\end{itemize}


The structure of this paper is as follows. First, in \autoref{s2}, we give some basic preliminaries on  \'etale groupoids and their topological full groups. In \autoref{s3} we prove our main results, Theorem \ref{thmintro:equiv} and \ref{thmintro:Realization} in \autoref{s3last}, and we derive general consequences from them. Finally, in \autoref{s5} we give examples and applications.

\vspace{0.2cm}

\textbf{Acknowledgements:} The authors thank James Belk, Collin Bleak, Jan Gundelach, David Kerr, Xin Li and Eduardo Scarparo for useful comments throughout the development of this project. We also thank Matt Brin for sending us valuable feedback once a first
version of this manuscript circulated, as well as the referee for pointing out a mistake in an earlier version and for kindly
providing \autoref{eg:referee}.
This paper forms a part of the PhD thesis of the second named author at the University of Glasgow \cite{mythesis}.

\section{Preliminaries}
\label{s2}
In this section, we collect a number of basic definitions and 
results that will be used throughout our work. For further discussion on groupoids, we 
refer the reader to \cite{simsgpds}. 

\begin{df}
A \textit{groupoid} is a small category where every morphism is an isomorphism. More explicitly, a groupoid is a set $\G$ 
with partially defined multiplication and everywhere 
defined involutive operation 
$\gamma \mapsto \gamma ^{-1}$, satisfying:
\begin{enumerate}
    \item Associativity: If $\gamma_1\gamma_2$ and $(\gamma_1\gamma_2)\gamma_3$ are defined, then $\gamma_2\gamma_3$ is defined and $(\gamma_1\gamma_2)\gamma_3=\gamma_1(\gamma_2\gamma_3)$
    \item For every $\gamma\in \G$, the product $\gamma\gamma^{-1}$ is always defined. 
    \item For $\gamma_1,\gamma_2\in \G$, if $\gamma_1\gamma_2$ is 
    defined, then 
    \[\gamma_1=\gamma_1\gamma_2 \gamma_2^{-1} \ \ \mbox{ and } \ \ \gamma_2=\gamma_1^{-1} \gamma_1 \gamma_2.\]
\end{enumerate}
\end{df}
Elements of the form $\gamma\gamma^{-1}$ are called \emph{units}, and the space of units in $\G$ is denoted $\G^{(0)}$. The subset of $\G\times\G$ consisting of composable pairs is denoted by $\Gn{2}$. We define the \emph{source} and \emph{range} maps 
$s,r\colon \G\to \G^{(0)}$ by $s(\gamma)=\gamma^{-1}\gamma$ and
$r(\gamma)=\gamma\gamma^{-1}$ for all $\gamma\in\G$.

A \textit{topological groupoid} is a groupoid endowed with a topology such that the multiplication and inverse operations are continuous. We assume further that the topology on our groupoid is Hausdorff.
An (open) \emph{bisection} in $\G$ is an open subset $B\subseteq\G$ for which
the restrictions of $s$ and $r$ are homeomorphisms and $s(B), r(B)$ are open. We will work exclusively with ample 
groupoids in this paper, which we define next.

\begin{df}
Let $\G$ be a topological groupoid. 
\begin{enumerate} 
 \item We say that $\G$ is \emph{\'etale}
if its source and range maps are local homeomorphisms.
Equivalently, the set of all open bisections in $\G$ forms a 
basis for the topology of $\G$. 
\item We say that it is \emph{ample} if 
the set of all compact, open bisections in $\G$ forms a basis 
for the topology of $\G$. Equivalently, $\G$ is \'etale and 
$\G^{(0)}$ is totally disconnected
\item We say that $\G$ is a \emph{Cantor groupoid} 
if $\G^{(0)}$ is a Cantor space. 
\end{enumerate}

For an ample groupoid $\G$, we will write $\mathcal{B}(\G)$ for 
the set of its compact open bisections. In this case, 
$\B(\G)$ is an inverse monoid with respect to set 
multiplication and inverses.
\end{df}

Note that \'etale, Cantor groupoids are ample.
We will often need to assume minimality and essential principality for our groupoids, so we define these.


\begin{df}
Let $\G$ be a topological groupoid, and let $x\in \Gn{0}$.
\begin{enumerate}
 \item We denote by $\G x$ the $\G$-\emph{orbit} of $x$, which is
 given by 
 \[\G x=\{r(\gamma)\colon \gamma\in \G, s(\gamma)=x\}.\]
 \item We say that $\G$ is \emph{minimal} if $\G x$ is dense in $\Gn{0}$ for all $x \in \Gn{0}$. 
 \item We denote by $\G_x$ the \emph{isotropy group} of $x$, which
 is given by 
 \[\G_x=\{\gamma\in\G\colon r(\gamma)=s(\gamma)=x\}.\]
 \item We say that $\G$ is \emph{principal} if $\G_x=\{x\}$ for 
 every $x\in \G^{(0)}$. 
 \item We say that $\G$ is \emph{essentially principal} if the set $ \big\{x\in \G^{(0)}\colon \G_x=\{x\}\big\}$ is dense in $\G^{(0)}$. 
 \item We say that $\G$ is \emph{effective} if for all $x \in \G^{(0)}$, all $\gamma \in \G_x \setminus \{ x \}$, and all open compact bisections $B $ such that $\gamma \in B$, there exists $\gamma' \in B$ such that $s(\gamma') \neq r(\gamma')$
\end{enumerate}
\end{df}

Some references in the literature define essential principality using the condition that the units with trivial isotropy have full measure for every invariant measure on the unit space. This notion is closely related to effectiveness, and the two agree when $\G$ is Hausdorff. In this text, we have a blanket assumption of $\G$ being effective, since this is what allows us to think of open compact bisections as local homeomorphisms without losing information.  

We now give a key example for this paper: the transformation groupoid associated to a topological 
dynamical system. 

\begin{eg}\label{eg:TransfGpd}
Let $\Gamma \acts X$ be a discrete group acting by homeomorphisms on a topological space $X$. Then we define the \textit{transformation groupoid} $\Gamma \ltimes X$ to be $\Gamma\ltimes X$ as a set, 
with multiplication determined by $(g,h(x))(h,x)=(gh,x)$, and 
inverses given by $(g,x)=(g^{-1},g(x))$. Here $s(g,x)=(1,x)$ and $r(g,x)=(1,g(x))$, and the unit space is canonically identified with $X$. Moreover, $\Gamma\ltimes X$ is \'etale if and only if 
$\Gamma$ is discrete. Many properties of $\Gamma \acts X$ translate into properties of the transformation groupoid, for example:
\begin{enumerate}
\item The action is \textit{free}, namely for all $g \in \Gamma\setminus\{1\}$ and $x \in X$, one has $gx\neq x$,
if and only if $\Gamma\ltimes X$ is principal. 
\item The action is \textit{topologically free}, namely for all $g \in \Gamma\setminus\{1\}$, the set $\{x \in X \colon gx=x\}$ has empty interior, if and only if $\Gamma\ltimes X$ is essentially principal. 
\item The action is minimal, namely for each $x\in X$, the orbit $\Gamma x$ is dense in $X$, if and only if the groupoid is minimal. 
\item A basis for the (compact) open bisections consists of sets of the form $(g,U)$, where $g \in \Gamma$ and $U\subseteq X$ is (compact and) open. In particular, $\Gamma\ltimes X$ is ample if and only if $\Gamma$ is
discrete and $X$ is totally disconnected.
\end{enumerate}
\end{eg}

Crucial in our paper is the construction of a groupoid of germs.
\begin{eg}[Groupoid of Germs]
  Let $\Gamma \acts X$ be an action as in \autoref{eg:TransfGpd}. The associated \emph{groupoid of germs} $\text{Germ}(\Gamma\acts X)$ is defined as
\[\text{Germ}(\Gamma\acts X):= \Gamma \ltimes X /\sim,\]
where $\sim$ is the equivalence relation given by $(g,x) \sim (h,y)$ whenever $x=y$ and there exists a neighbourhood $U$ of $x$ such that $g|_U=h|_U$. 
The groupoid structure, as well as its topology, are inherited from 
$\Gamma\ltimes X$. We warn the reader that $\mathrm{Germ}(\Gamma\acts X)$ may fail to be Hausdorff in general (although $\Gamma\ltimes X$ always is). 
\end{eg}

We are ready to define the topological full group of a groupoid. This concept was first introduced in the setting of Cantor minimal systems by Giordano-Putnam-Skau in \cite{giordano1999full}, and later generalised by Matui to more general \'etale groupoids in \cite{Mat_homology_2012}. These groups have been extensively studied in recent years, among others by Matui \cite{Mat_etale_2016, Mat_topological_2015} and Nekrashevych \cite{Nek_simple_2019}. The definition below, allowing for locally compact unit spaces, is due to Nyland-Ortega \cite{NylOrt_matuis_2021}.

\begin{df}\label{groupoid tfg}
Let $\G$ be an ample, effective groupoid. For each compact open subset 
$K \subseteq \G^{(0)}$, set 
\[ \mathsf{F}(\G)_K=\big\{ B \in \mathcal{B}(\G)\colon s(B)=r(B)=K\big\} \]
Note that $\mathsf{F}(\G)_K$ is a group, with unit given by the (trivial) bisection $K$. There is a canonical action $\alpha^K$
of $\mathsf{F}(\G)_K$ on $K$ by homeomorphisms. 
Extending these homeomorphisms as the identity outside of $K$ yields an action of $\mathsf{F}(\G)_K$ on $\G^{(0)}$, which we also
denote by $\alpha^K$. If $L\subseteq\G^{(0)}$ is another compact
open subset with $K\subseteq L$, it is then easy to see 
that $\alpha^K(\mathsf{F}(\G)_K) \subseteq \alpha^L(\mathsf{F}(\G)_L)$. 

The \emph{topological full group} of $\G$ is then the inductive limit
\[ \mathsf{F}(\G)=\varinjlim_{K\subseteq \G^{(0)}} \alpha^K(\mathsf{F}(\G)_K) \ \subseteq \mathrm{Homeo}(\G^{(0)}),\]
where $K$ ranges over all compact open subsets of $\G^{(0)}$.
\end{df}

When $\G^{(0)}$ is compact, we have $\mathsf{F}(\G)=\mathsf{F}(\G)_{ \mathcal{G}^{(0)}}$. In particular, for \'etale Cantor groupoids we have:
\[\mathsf{F}(\G)=\big\{B \in \mathcal{B}(\G)\colon s(B)=r(B)=\G^{(0)} \big\}.\]

One can think of the topological full group as the subgroup of homeomorphisms of $\G^{(0)}$ that locally look like bisections. 
In fact, an equivalent definition is the following:
\[\mathsf{F}(\G)=\left\{ g \in \mathrm{Homeo}(\G^{(0)})\colon \begin{aligned} \mbox{for all } x \in \G^{(0)} \mbox{ there is } B \in \mathcal{B}(\G)\\  
\mbox{with } x \in s(B), \text{ and } g|_{s(B)}=B \ \ \end{aligned}\ \right\}.\]                                                                                                                
In other words, the elements of the
topological full group are global symmetries (homeomorphisms) built from local symmetries (partial homeomorphisms) induced by compact open bisections.

For transformation groupoids, the following convenient description
of the topological full group goes back to the seminal work of 
Giordano-Putnam-Skau \cite{giordano1999full}:

\begin{eg}\label{TFGdefgroupaction}
Let $\Gamma$ be a discrete group, let $X$ be a locally
compact Hausdorff space, and let $\Gamma \acts X$ be a topologically free action by homeomorphisms. 
Then $\mathsf{F}(\Gamma\ltimes X)$ can be naturally identified with the subgroup of homeomorphisms $g \in \mathrm{Homeo}(X)$ for which there exists a continuous map $f_g \colon X \to \Gamma $ such that $g(x)=f_g(x)\cdot x$ for all $ x \in X$. Such a map $f_g$ is called an \emph{orbit cocycle} for $g$.  

For the associated groupoid of germs, it follows from 
\cite[Proposition 4.6]{NylOrt_topological_2019} that
$\mathsf{F}(\mathrm{Germ}(\Gamma\acts X))$ agrees with $\mathsf{F}(\Gamma \ltimes X)$, so it can be computed using orbit cocycles for $\Gamma\acts X$ as in the 
previous paragraph.
\end{eg}

In \cite{Nek_simple_2019}, Nekrashevych introduced the \emph{alternating subgroup} $\mathsf{A}(\G)$ of $\mathsf{F}(\G)$, in analogy with the alternating subgroup $A_n$ of $S_n$, defined via so-called multisections. In this paper, since all of 
our orbits are infinite, we have the luxury of a much simpler definition. 

\begin{df}\label{df:A(G)}
Let $\G$ be an ample groupoid such that all orbits are infinite. Its \emph{alternating group} 
$\mathsf{A}(\G)$ is the subgroup of $\mathsf{F}(\G)$ generated by the elements of order three coming from multisections, that is, elements $g\in\mathsf{F}(\G)$ of order three such that there exist disjoint compact open subsets $X_0,X_1,X_2 \subseteq \G^{(0)}$ satisfying $g(X_j)=X_{j+1}$ for $j=0,1,2$ (with indices taken modulo 3), and such that $g$ is the identity on $\G^{(0)} \setminus (X_0 \sqcup X_1 \sqcup X_2)$. 
\end{df}

When $\G$ is purely infinite and minimal, the definition given above agrees with Nekrashevych's definition by  \cite[Proposition~3.6]{Nek_simple_2019}, and thus $\mathsf{A}(\G)$ is normal in $\mathsf{F}(\G)$. 


Another subgroup of $\mathsf{F}(\G)$ considered by Nekrashevych is
the derived subgroup:

\begin{df}
Let $\G$ be an ample groupoid. Then the \emph{commutator subgroup} (also known as the \emph{derived subgroup}) $\mathsf{D}(\G)$ is the subgroup of $\mathsf{F}(\G)$ generated by the commutators $[g,f]$, for $g,f \in \mathsf{F}(\G)$. 
\end{df}

It is immediate from the definition that $\mathsf{D}(\G)$ is
normal in $\mathsf{F}(\G)$. One can also show that $\mathsf{A}(\G)\subseteq \mathsf{D}(\G)$; see \cite[Theorem~4.1]{Nek_simple_2019}. As it turns out, it is often the case that $\mathsf{A}(\G)=\mathsf{D}(\G)$, and 
this is in particular always the case for minimal, purely infinite groupoids (see \autoref{minimal essentially principal then simple}), which is the main focus of this work. In fact, there is 
presently no known example of an \'etale groupoid that does not have this property. 
In the rest of this work, we will mostly focus on alternating groups for technical reasons, but the reader should keep in mind that there is no distinction between them for purely 
infinite groupoids.

\section{Characterisations of Pure Infiniteness and Minimality}
\label{s3}

We begin by recalling the notion of pure infiniteness for
\'etale groupoids due to Matui \cite{Mat_topological_2015}.
All groupoids in this section are effective and ample.

\begin{df}\label{df:PI}
A groupoid $\G$ is said to be \emph{purely infinite} if for every compact and open set $X \subseteq \G^{(0)}$ there exist compact open bisections $B,B'\subseteq \G$ such that $s(B)=s(B')=X$, and $r(B),r(B')$ are disjoint and contained in $X$. 
\end{df}


The terminology above is justified by the fact that 
the reduced groupoid C*-algebra of a minimal, 
essentially principal, purely infinite \'etale groupoid is 
purely infinite and simple. The following 
result is known to the experts \cite{Mat_topological_2015}, but we provide a proof nonetheless for 
the convenience of the reader. 
Recall that a projection in a C*-algebra is said to be 
\emph{properly infinite} if it is Murray-von Neumann equivalent
to two pairwise orthogonal subprojections of itself.

\begin{prop}\label{prop:PIgroupoidC*alg}
Let $\mathcal{G}$ be a minimal, 
essentially principal, purely infinite, ample groupoid. Then $C^*_r(\mathcal{G})$ is purely 
infinite and simple.
\end{prop}
\begin{proof}
Simplicity of $C^*_r(\G)$ is well-known, so we only prove
pure infiniteness.
By \cite[Theorem~B]{BonLi_ideal_2020}, it suffices to show 
that every projection in $C_0(\G^{(0)})$ is properly infinite
in $C^*_r(\G)$. Let $p\in C_0(\G^{(0)})$ and let $X$ denote its 
support, which is compact and open.
Use pure infiniteness of $\G$ 
to obtain compact open bisections $B,B'\subseteq \G$ 
satisfying $s(B)=s(B')=X$ and $r(B)\sqcup r(B')\subseteq X$. 
Set $v=\mathbbm{1}_B$ and $w=\mathbbm{1}_{B'}$, which belong to 
$C_c(\mathcal{G})$ since $B$ and $B'$ are compact. 
The conditions on $B$ and $B'$ easily give 
\[v^*v=w^*w=p, \ \ \ vv^*\perp ww^*, \ \mbox{ and } \
 vv^*,ww^*\leq p.
\]
Thus $vv^*$ and $ww^*$ are pairwise orthogonal subprojections of $p$ which are Murray-von Neumann equivalent to $p$. We conclude that $p$ is properly
infinite, and hence that $C^*_r(\G)$ is purely infinite, as desired.
\end{proof}

Note that many ample groupoids are known to be purely infinite and minimal, for example: 
\begin{enumerate}
    \item Transformation groupoids of certain nonamenable groups acting amenably, such as those studied in \cite{GarGefKraNar_classifiability_2023,elek2020amenable, rordam2012purely}.
    \item Shift of finite type groupoids \cite{Mat_topological_2015}, or more generally, large classes of graph groupoids \cite{NylOrt_topological_2019}. 
    \item Groupoids arising from Beta expansions \cite{MatMat_topological_2014}. 
    \item Certain groupoids arising from left regular representations of Garside categories \cite{li2021left}; see also \cite[Theorem~A]{li2022cstar}. 
\end{enumerate}
We will need the following result of Matui:

\begin{lma}\label{lma:PIcompact and open}(\cite[Proposition 4.11]{Mat_topological_2015}).
Let $\G$ be an essentially principal, \'etale, ample groupoid. The following are equivalent:
\begin{enumerate}
    \item $\G$ is purely infinite and minimal.
    \item For all compact and open subsets $X,Y \subseteq \G^{(0)}$, there exists a compact open bisection $B\subseteq \G$ such that $s(B)=X$ and  $r(B) \subseteq Y$.
    \item For all compact and open subsets $X,Y\subseteq \G^{(0)}$
    with $X\neq \G^{(0)}$ and $Y\neq \emptyset$,
    there exists $\alpha\in \mathsf{F}(\G)$ such that 
    $\alpha(X)\subseteq Y$.
\end{enumerate}
\end{lma}

Next, we record the fact, essentially due to Matui, 
that the subgroups $\mathsf{A}(\G)$ and $\mathsf{D}(\G)$ agree for the class of 
purely infinite, minimal groupoids.

\begin{prop}\label{minimal essentially principal then simple}
Let $\G$ be an ample, effective, purely infinite and minimal groupoid.
Then $\mathsf{D}(\G)$ is simple and thus $\mathsf{A}(\G)=\mathsf{D}(\G)$.
\end{prop}
\begin{proof}
Simplicity of $\mathsf{D}(\G)$ follows from \cite[Theorem~4.16]{Mat_etale_2016}. Thus $\mathsf{A}(\G)=\mathsf{D}(\G)$ by 
\cite[Theorem~4.1]{Nek_simple_2019}.
\end{proof}

\subsection{Embeddings of Thompson's group \texorpdfstring{$V$}{V}}

In \cite[Proposition~4.10]{Mat_topological_2015},
Matui proved that if $\G$ is purely infinite, then $\mathsf{F}(\G)$ contains $\mathbb{Z}_2 \ast \mathbb{Z}_3$ as a subgroup (in fact, for this it suffices for $\G$ to be 
\emph{properly infinite}). In
particular, $\mathsf{F}(\G)$ is nonamenable. We will strengthen this
result by showing that the Thompson group $V$ \emph{always} embeds into
$\mathsf{F}(\G)$. In fact, we give a characterisation of 
pure infiniteness of $\G$ in terms of the 
existence of certain embeddings of $V$ into $\mathsf{F}(\G)$; 
see \autoref{thm:Vembeds}. 

We need some preparation first. The 
following notion is due to Bleak, Elliott and Hyde \cite{BleEllHyd_sufficient_2020}.

\begin{df}\label{df:Vigorous}
(\cite[Definition 1.1]{BleEllHyd_sufficient_2020}). Let $\Gamma \leq \mathrm{Homeo}(\mathcal{C})$ be a subgroup of homeomorphisms of the Cantor space $\mathcal{C}$. We say that $\Gamma$ is \emph{vigorous}\footnote{It would be 
more technically correct to say that the action of $\Gamma$ on $\mathcal{C}$ is vigorous, since vigor is not a property of $\Gamma$, but rather of the way it sits in $\mathrm{Homeo}(\mathcal{C})$.} if whenever 
$X, Y_1,Y_2 \subseteq \mathcal{C}$ are compact and open with $Y_1\neq X$ and $Y_2\neq \emptyset$, and satisfy $Y_1,Y_2 \subseteq X$, then there exists $g \in \Gamma$ such that $g$ is the identity on $\mathcal{C} \setminus X$, and $g(Y_1) \subseteq Y_2$. 
\end{df}

We will need the following observations.

\begin{prop}\label{lemma:NoInvMeas}
Let $\Gamma \leq \mathrm{Homeo}(\mathcal{C})$ be a vigorous group. Then: 
\begin{enumerate}
 \item There
is no $\Gamma$-invariant Borel probability measure on $\mathcal{C}$. In particular,
$\Gamma$ is not amenable.
\item $\Gamma$ has infinite conjugacy classes (ICC).
\end{enumerate}
\end{prop}
\begin{proof}
(1). Let $\mu$ be a Borel probability measure on $\mathcal{C}$. Choose nontrivial 
compact and open
sets $Y_1,Y_2\subseteq \mathcal{C}$ with $\mu(Y_1)<\mu(Y_2)$. By vigor
(with $X=\mathcal{C}$),
there is $g\in \Gamma$ with $g(Y_1)\subseteq Y_2$. Then
$\mu(gY_2)\leq \mu(Y_1)< \mu(Y_2)$, so $g\cdot\mu\neq\mu$ 
and hence $\mu$ is not $\Gamma$-invariant.

The last claim follows since a discrete group is amenable if and 
only if every action on a compact Hausdorff space admits an invariant Borel probability measure.

(2). Let $g\in \Gamma\setminus\{1\}$; we will show that 
the conjugacy class of $g$ is infinite.
Find a compact and open subset $Y_1\subseteq \mathcal{C}$ such that  $g(Y_1) \cap Y_1= \emptyset$, and
find an infinite family $\{Y_2^{(n)}\colon n\in\N\}$ of pairwise disjoint 
compact and open subsets of $Y_1$. 
Use vigor of $\Gamma$ with $X=\mathcal{C}$ to find, for every $n\in\N$,
a group element $h_n\in \Gamma$ with support contained in $\mathcal{C}\setminus Y_1$, 
such that $h_n(g(Y_1)) \subseteq Y_2^{(n)}$. 
One readily checks that $h_n^{-1}gh_n(Y_1)\subseteq Y_2^{(n)}$ for all $n\in\N$; in particular, this implies that 
$h_n^{-1}gh_n\neq h_m^{-1}gh_m$ if $n\neq m$, as 
desired.\end{proof}

Note that in the following theorem, we must
exclude $X=\G^{(0)}$, since the existence of an embedding 
$V\hookrightarrow \mathsf{F}(\G)$ with full support implies that 
$\G^{(0)}$ has trivial homology class. 

\begin{prop}\label{thm:Vembeds}
Let $\G$ be an \'etale Cantor groupoid which is purely infinite and minimal.
Then for every nontrivial compact and open subset $X\subseteq \G^{(0)}$, there exists
an embedding $\varphi_X\colon V \hookrightarrow \mathsf{A}(\G)$ 
such that $\mathrm{supp}(\varphi_X(V))$ contains $X$.
\end{prop}
\begin{proof}
Let $X \subseteq \G^{(0)}$ be nontrivial, compact and open. 
Use \autoref{lma:PIcompact and open} to find a compact and open bisection $B$
such that $s(B)=\G^{(0)}$ and $r(B) \subsetneq \G^{(0)}\setminus X$. Set $Y=\G^{(0)}\setminus r(B)$, which is a compact and open subset 
of $\G^{(0)}$ containing $X$. Moreover, 
\[\![Y]=[\G^{(0)}]-[r(B)]=[s(B)]-[r(B)]=0.\]

Denote by $\G |_Y$ the \emph{reduction} of $\G$ to $Y$, namely
\[\G |_Y=r^{-1}(Y)\cap s^{-1}(Y),\]
which is an \'etale subgroupoid of $\G$ with unit space 
equal to $Y$.
Since $Y$ is itself nonempty and compact and open,
there is an embedding $\mathsf{F}(\G_Y)\hookrightarrow \mathsf{F}(\G)$
whose support is exactly $Y$. 
Denote by $\mathcal{E}_2$ the groupoid associated to the 
shift of finite type corresponding to the single vertex graph with 2 loops.
(This groupoid is also called the \emph{Cuntz groupoid}, since its C*-algebra is
canonically isomorphic to the Cuntz algebra $\mathcal{O}_2$.) 
Recall that $\mathsf{F}(\mathcal{E}_2)\cong V$ and that $H_0(\mathcal{E}_2)=\{0\}$. 
Since the class of
the unit space of $\G |_Y$ is trivial in homology, 
by \cite[Proposition~5.14]{Mat_etale_2016} 
there exists a unital homomorphism $\pi\colon C^*_r(\mathcal{E}_2)\to C^*_r(\mathcal{G}_Y)$ satisfying
\begin{enumerate}
 \item[(a)] $\pi(C(\mathcal{E}_2^{(0)}))\subseteq C(Y)=C(\mathcal{G}_Y^{(0)})$, and 
 \item[(b)] for every compact, open bisection $B\subseteq \mathcal{E}_2$, there exists
 a compact, open bisection $B'\subseteq \G_Y$ such that $\pi(\mathbbm{1}_{B})=\mathbbm{1}_{B'}$. 
\end{enumerate}
Since $C^*_r(\mathcal{E}_2)\cong \mathcal{O}_2$ is simple, the map $\pi$ is
injective. In particular, $\pi$ induces an embedding
\[\varphi_X\colon V\cong \mathsf{F}(\mathcal{E}_2) \hookrightarrow \mathsf{F}(\G_Y) \subseteq  \mathsf{F}(\G).\]

It remains to show that the image of $\varphi_X$
is actually
contained in $\mathsf{A}(\G_Y)$. By \autoref{minimal essentially principal then simple}, it suffices to show that the image of this embedding is
contained in $\mathsf{D}(\G_Y)=\big[\mathsf{F}(\G_Y),\mathsf{F}(\G_Y)\big]$. Since group homomorphisms map commutators to commutators, and since $V$ is equal to its own commutator (because it is simple), it follows that the image of the embedding $V \hookrightarrow \mathsf{F}(\G|_Y )$ is contained in $\mathsf{D}(\G_Y)=\mathsf{A}(\G_Y)$.\end{proof}

The last statement in \autoref{thm:Vembeds} is not equivalent
to the remaining ones, that is, having
sufficiently many embeddings of $V$ into the topological full group
does not imply pure infiniteness of the groupoid. For example, if we let Thompson's group $V$ act trivially on
the Cantor space and denote by $\G$ the associated 
transformation groupoid, then $\G$ is not purely infinite
even though $V$ embeds into $\mathsf{F}(\G)=V$ with full support. 

Also, we do not seem to be able to obtain much information about the 
way that $V$ acts on $\mathcal{G}^{(0)}$ through these embeddings.
The next example, which was provided to us by the referee, shows that 
$V$ does not act vigorously on any parts of $\mathcal{G}^{(0)}$.

\begin{eg}\label{eg:referee}
Let the trivial group $\mathbf{1}$ act (trivially) on the natural 
numbers $\mathbb{N}$, and let $\N V_{\mathbf{1}}$ denote the associated twisted higher-dimensional Brin-Thompson group defined in 
\cite[Section~1]{belk2022twisted}. Writing $\mathcal{C}$ for the Cantor space, we consider the natural action of $\N V_{\mathbf{1}}$ on $\mathcal{C}^{\mathbb{N}}$. 
By \cite[Theorem~A]{belk2022twisted}, this group 
is not finitely generated. It is clear from the definition of 
$\mathbb{N}V_{\mathbf{1}}$ that this group agrees with the ascending union $\bigcup_{n\in\mathbb{N}}nV$ of the higher-dimensional 
Thompson-Brin groups. For every $n\in\mathbb{N}$, it is easy to see 
that the restriction of the action $\N V_{\mathbf{1}}\curvearrowright \mathcal{C}^{\mathbb{N}}$ to $nV$ is the product of the canonical
action of $nV$ on $\mathcal{C}^{\{1,\ldots,n\}}$ and the trivial action on the remaining coordinates. In particular, this action is not vigorous. Since every finitely generated subgroup of $\N V_{\mathbf{1}}$ must be contained in $nV$ for some $n\in\N$, we deduce that no
finitely generated subgroup of $\N V_{\mathbf{1}}$ acts vigorously
on $\mathcal{C}^{\N}$. By letting $\mathcal{G}$ denote the groupoid of germs of the action $\N V_{\mathbf{1}}\curvearrowright \mathcal{C}^{\mathbb{N}}$, we conclude that $\mathcal{G}$ satisfies
the assumptions of \autoref{thm:Vembeds} but the embeddings of $V$
into its topological full group provided by said proposition do not act vigorously on their supports.
\end{eg}

The following is a curious feature of the example described above,
which suggests that it is challenging to detect pure infiniteness
and minimality of a groupoid by looking at finitely generated 
subgroups of its topological full group.

\begin{rem}
Let the notation be as in \autoref{eg:referee}, and write 
$\mathcal{G}$ for the transformation groupoid of the action $\N V_{\mathbf{1}}\curvearrowright \mathcal{C}^{\mathbb{N}}$. Then 
$\mathcal{G}$ is purely infinite and minimal. On the other hand, if 
$G\leq \mathbb{N}V_{\mathbf{1}}$ is any finitely generated subgroup,
then the transformation groupoid associated to the restriction to $G$
of $\N V_{\mathbf{1}}\curvearrowright \mathcal{C}^{\mathbb{N}}$
is no longer purely infinite and minimal. 
\end{rem}

We now derive some conclusions regarding the (generalised) word problem for topological full groups. 

\begin{df}
\label{wordproblem} Let $\Gamma$ be a finitely generated group with finite generating set $S$.
We say that $\Gamma$ has 
\begin{enumerate}
    \item \emph{solvable word problem} (with respect to $S$), if given a finite word $w\in S^n$ there exists an algorithm that determines, in finitely many steps, whether or not $w$ is the identity.
     \item \emph{solvable generalised word problem} (with respect to $S$), if for all subgroups $\Lambda \subseteq \Gamma$, given a finite word $w\in S^n$, there exists an algorithm that determines, in finitely many steps, whether $w$ belongs to $\Lambda$.  
\end{enumerate}
\end{df}

Knowing that topological full groups contain $V$ allows us to deduce that 
they have unsolvable generalised word problem:

\begin{cor}
 Let $\G$ be a minimal, essentially principal, ample, purely infinite groupoid. Then $\mathsf{F}(\G)$ and $\mathsf{A}(\G)$ do not have a solvable generalised word problem. 
 \end{cor}
\begin{proof}
Note that 
unsolvable generalised word problems are inherited by containing groups. Thus, it suffices to show the statement for the alternating group. In turn, since $\mathsf{A}(\G)$ contains $V$ by \autoref{thm:Vembeds}, it suffices to argue that $V$ has an unsolvable generalized word problem; this follows from \cite{russian}. 
\end{proof}

In contrast to the above corollary, in many cases $\mathsf{F}(\G)$ 
\emph{does} have a solvable word problem, for example the topological full groups of many groupoids that arise as left regular representations of Garside Categories (see \cite[Corollary~C]{li2021left}) are known to have solvable word problem. Therefore, we obtain infinitely many nonisomorphic groups spanning many families that demonstrate this boundary between the word problem and the generalised word problem. 

\subsection{Vigor}
Note that $\mathsf{A}(\G)$ is vigorous whenever $\mathcal{G}$ is an \'etal, Cantor, purely infinite minimal groupoid, by \autoref{thm:Vembeds}. 
In fact, vigor of the derived subgroup is \textit{equivalent} to the groupoid being purely infinite and minimal even in the effective case. 

\begin{prop}\label{lma:DGvigor}\label{purely infinite if and only if vigor} 
Let $\G$ be an \'etale, effective Cantor groupoid. 
Then the following are equivalent:
\begin{enumerate}
    \item $\G$ is purely infinite and minimal
    \item One (equivalently, both) of $\mathsf{A}(\G)$ or $\mathsf{F}(\G)$ is vigorous.
\end{enumerate}  
\end{prop}
\begin{proof}
Note that vigor of $\mathsf{A}(\G)$ implies vigor of $\mathsf{F}(\G)$. Thus, it suffices
to show that (1) is equivalent to vigor of $\mathsf{A}(\G)\leq \mathrm{Homeo}(\G^{(0)})$.

(1) implies (2).
This proof is entirely constructive. Let $X, Y_1,Y_2 \subseteq \G^{(0)} $ be compact and open sets with $Y_1,Y_2 \subseteq X$, $Y_1\neq X$ and $Y_2\neq \emptyset$. We aim to find a multisection $\alpha\in\mathsf{F}(\G)$ of order 3 such that $\alpha$ is the identity outside of $X$ and takes $Y_1$ inside $Y_2$. Write $Y_2 \setminus Y_1$ as a nontrivial disjoint union $Y_2 \setminus Y_1=Z_{2,1}\sqcup Z_{2,2}$ of compact and open sets. Since $\G$ is minimal and purely infinite,
part~(2) of \autoref{lma:PIcompact and open} provides us with a compact open bisection $B_1\subseteq \G$ with $s(B_1)=Y_1$ and $r(B_1) \subseteq Z_{2,1}$. Use
part~(2) of \autoref{lma:PIcompact and open}
again to find a compact open bisection $B_2$ with $s(B_2)=r(B_1) \subseteq Z_{2,1}$ and $r(B_2) \subseteq Z_{2,2}$. Note that $(B_2 B_1)^{-1} $ is also a bisection.  Set 
\[\alpha=\big((B_2B_1)^{-1} \cup B_1 \cup B_2\big) \cup \big(\G^{(0)} \setminus s(B_1) \cup r(B_1) \cup r(B_2)\big).\]
One readily checks that $\alpha$ is a full bisection, 
so it defines an element of $\mathsf{F}(\G)$. 
It is also clear that it defines a multisection in the sense of \autoref{df:A(G)}, since 
it satisfies
\[\alpha(Y_1)=Z_{2,1}, \ \ \alpha(Z_{2,1})=Z_{2,2}, \ \ \mbox{ and } \ \ \alpha(Z_{2,2})=Y_1,\]
while $\alpha$ acts trivially on the rest of $\G^{(0)} $.
Since 
$\alpha$ has order 3, it follows that $\alpha\in \mathsf{A}(\G)$.
Note that $\alpha$ is the identity outside of $X$, and that $\alpha(Y_1)= r(B_1) \subseteq  Z_{2,1} \subseteq Y_2$. This shows that $\mathsf{A}(\G)$ is vigorous.

(2) implies (1). 
We check condition~(2) of \autoref{lma:PIcompact and open}. 
Let $Y_1,Y_2$ be nonempty compact and open subsets of $\G^{(0)}$ with $Y_1\neq \mathcal{C}$. Use vigor with $X=\mathcal{C}$
to find $g \in \mathsf{A}(\G)$ with $\mathrm{supp}(g)=Y_1 \cup Y_2$ and $g(Y_1) \subseteq Y_2$. If $B\subseteq \G$ denotes the compact
open full bisection determining $g$, then $B|_{Y_1}$ is a compact
open bisection satisfying $s(B|_{Y_1})=Y_1$ and $r(B|_{Y_1})\subseteq Y_2$, as desired.  
\end{proof}

As an immediate consequence, we show that 
$\mathsf{F}(\G)$ enjoys a strengthening of the ICC 
property with respect to $\mathsf{D}(\G)$; this will be needed later.

\begin{cor}\label{cor:RelativeICC}
Let $\G$ be a minimal, purely infinite, \'etale Cantor groupoid. Then the $\mathsf{D}(\G)$-conjugacy class of every 
nontrivial element of $\mathsf{F}(\G)$ is infinite.
\end{cor}
\begin{proof}
This proof is similar to that of 
\autoref{lemma:NoInvMeas}. 
Let $g\in \mathsf{F}(\G)\setminus\{1\}$.
Find a compact and open subset $Y_1\subseteq\G^{(0)}$ such that  $g(Y_1) \cap Y_1= \emptyset$, and
find an infinite family $\{Y_2^{(n)}\colon n\in\N\}$ of pairwise disjoint 
compact and open subsets of $Y_1$. 
Use vigor of $\mathsf{A}(\G)$ to find, for every $n\in\N$,
a group element $h_n\in \mathsf{A}(\G)$ with support contained in $\mathcal{G}^{(0)}\setminus Y_1$, 
such that $h_n(g(Y_1)) \subseteq Y_2^{(n)}$. 
One readily checks that 
$h_n^{-1}gh_n\neq h_m^{-1}gh_m$ if $n\neq m$, as they map $Y_1$ to different subsets (namely $Y_2^{(n)}$ and $Y_2^{(m)}$, respectively). This finishes the proof.\end{proof}

Expansivity is a condition for \'etale groupoids which
was introduced by Nekrashevych in \cite{Nek_simple_2019}
in order to show that certain alternating groups
in topological full groups are 2-generated. We
recall his notion below:
  
\begin{df}
Let $\G$ be an \'etale groupoid with Cantor unit space.
\label{expansive groupoid}
\begin{itemize}
    \item[(i)] A compact set $K \subseteq \G$ is said to be 
    \emph{generating} if $\G= \bigcup_{n \in \mathbb{N}}(K \cup K^{-1})^n$.
    \item[(ii)] A finite collection $\mathcal{B}$ of compact open bisections in $\G$ is said to be \emph{expansive} if $\bigcup_{n \in \mathbb{N}} (\mathcal{B} \cup \mathcal{B}^{-1})^n$ is a basis for the topology of $\G$. 
        \label{expansive nekrashevych}
\end{itemize}

We say that $\G$ is \emph{expansive} if there are a compact generating subset $K\subseteq \G$ and a finite cover $\mathcal{B}$ of $K$ which is expansive. 
\end{df}

We obtain the following consequence of our results 
in combination with \cite{Nek_simple_2019, BleEllHyd_sufficient_2020}. 

\begin{cor}\label{cor:twoGeneration}
Let $\G$ be an expansive, ample, essentially principal, minimal, purely infinite groupoid. Given $n>2$, there exist $g_1,g_2 \in \mathsf{A}(\G)$ such that $g_1$ has order $n$, 
the order of 
$g_2$ is finite, and $\mathsf{A}(\G)$ is generated by $\{g_1,g_2\}$. 
\end{cor}
\begin{proof}
Since $\G$ is minimal and essentially principal,
it follows from \cite[Theorem~1.1]{Nek_simple_2019} 
that $\mathsf{A}(\G)$ is simple. Moreover, since $\G$ is expansive,
\cite[Theorem~5.6]{Nek_simple_2019} implies 
that $\mathsf{A}(\G)=\mathsf{D}(\G)$ is finitely generated. By \autoref{lma:DGvigor}, $\mathsf{A}(\G)$ is vigorous, so the 
result follows from \cite[Theorem~1.12]{BleEllHyd_sufficient_2020}.
\end{proof}

\subsection{Compressibility}
The existence of compressible actions was shown to have relevance to the representation theory of Thompson-like groups in \cite{DudMed_finite_2014}. We recall the definition below:

\begin{df}\label{df:ComprAction}
An action of a discrete group $\Gamma$ on a locally compact Hausdorff space $X$ is said to be 
\emph{compressible}, if there exists a subbase $\mathcal{U}$ for the topology on $X$ such that: \begin{enumerate}
    \item for all $ g \in \Gamma$, there exists $U \in \mathcal{U}$ such that $\mathrm{supp}(g) \subseteq U$;
    \item for all $ U_1, U_2 \in \mathcal{U}$, there exists $g \in \Gamma$ such that $g(U_1) \subseteq U_2$;
    \item for all $ U_1,U_2,U_3 \in \mathcal{U}$ with $\overline{U}_1 \cap \overline{U}_2 = \emptyset$, there exists $g \in \Gamma$ such that $g(U_1) \cap U_3 = \emptyset$ and $\mathrm{supp}(g) \cap U_2= \emptyset$;
    \item for all $ U_1, U_2 \in \mathcal{U }$, there exists $U_3 \in \mathcal{U}$ such that $U_1 \cup U_2 \subseteq U_3 $.\end{enumerate}
\end{df}
\begin{rem}
Note that if $\Gamma \acts X$ is compressible, then $X$ cannot be compact. This is because $X$ cannot belong to $\mathcal{U}$ by (2), and at the same time $X$ cannot be written as a finite union of elements of $\mathcal{U}$ by (4). Therefore often when dealing with a group acting by homeomorphisms on the Cantor space, we consider point-stabiliser subgroups. 
\end{rem}
Compressibility and vigor are closely related notions. For example, the 
following is 
essentially a generalisation of the discussion in \cite[Subection~3.2]{DudMed_finite_2014}. 

\begin{lma}\label{lma:VigCompress}
Let $\mathcal{C}$ be a Cantor space and let $D \leq \mathrm{Homeo}(\mathcal{C})$ be a vigorous subgroup. For $x_0 \in \mathcal{C}$, set 
\[D_{x_0}:=\{ g \in D \colon \text{there exists a neighbourhood } Y \text{ of } x_0 \text{ such that } g|_Y=\id_Y\}.\] 
Then $D_{x_0}\curvearrowright \mathcal{C}\setminus\{x_0\}$ is compressible. \end{lma}
\begin{proof}
For an open set $Y\subseteq \mathcal{C}$, we set $D_Y=\{g\in D\colon g|_Y=\id_Y\}$.
Fix $x_0 \in \mathcal{C}$, and let $(Y_n)_{n\in\N}$ be a strictly decreasing sequence of compact and open subsets of $\mathcal{C}$ such that $\bigcap_{n=1}^\infty Y_n=\{x_0\}$. 
Then $D_{x_0}=\bigcup_{n=1}^\infty D_{Y_n}$. 

Let $\mathcal{U}$ be any basis of compact open subsets for the topology on $\mathcal{C} \setminus \{x_0\}$ which is closed under finite unions. 
We verify properties (1) through (4) in \autoref{df:ComprAction} for 
$D_{x_0}\curvearrowright \mathcal{C} \setminus \{x_0\}$ below.

(1). Let $g \in D_{x_0}$, and find a neighbourhood $Y$ of $x_0$ such that $g|_Y=\id_Y$. 
Find a basic open set $U\in \mathcal{U}$ such that $U \subseteq Y \setminus \{x_0\}$. Since 
$\mathcal{C}\setminus Y$ is compact, there exist $U_1,\ldots,U_n\in \mathcal{U}$ such
that the set $U=U_1\cup\ldots \cup U_n$ contains $\mathcal{C}\setminus Y$. Since 
$\mathcal{U}$ is closed under finite unions, it follows that $U$ belongs to $\mathcal{U}$. Since $g$ is supported on $U$, this shows (1).

(2). Let $U_1,U_2 \in \mathcal{U}$ be given. Since $x_0 \notin U_1, U_2$, this follows immediately by
using vigor of $D\curvearrowright \mathcal{C}$. 

(3). 
Let $U_1,U_2,U_3 \in \mathcal{U}$ satisfy $\overline{U}_1 \cap \overline{U}_2=\emptyset$. Since $U_1 \cup U_2 \cup U_3 \subseteq \mathcal{C} \setminus \{x_0\}$, we may find $n\in\N$ large enough so that $U_j \cap Y_n=\emptyset$ for $j=1,2,3$. Set $X=\overline{U_1} \cup Y_n \setminus Y_{n+1}$. Noting that $U_2 \cap X=\emptyset$ and that $X^c$ contains the neighbourhood $Y_{n+1}$ of $x_0$, we use vigor to find an element $g \in D_{x_0}$ such that 
\[g(U_1) \subseteq g(\overline{U_1}) \subseteq Y_n \setminus Y_{n+1} \subseteq Y_n\subseteq (U_1 \cup U_2 \cup U_3)^c \subseteq U_3^c.\]
Moreover, the support of $g$, which is a subset of $X$, contains a neighbourhood of $x_0$ and is contained in $U_2^c$. Thus $g \in D_{x_0}$ satisfies the required properties. 
    
(4). This follows by construction, as $\mathcal{U}$ is closed under finite unions. 
\end{proof}

We now obtain further characterisations of pure infiniteness for minimal groupoids:

\begin{lma}\label{lemma:EquivalencesCompressible}
Let $\G$ be an essentially principal, \'etale Cantor groupoid. Then, the following are equivalent:
\begin{enumerate}
    \item $\G$ is purely infinite and minimal.
    \item For all $x_0 \in \G^{(0)}$, the action $\mathsf{A}(\G)_{x_0} \acts \G^{(0)} \setminus \{x_0\}$ 
    is compressible  \label{compressible if and only if pureinf}
\end{enumerate}
\end{lma}
\begin{proof}
If $\G$ is purely infinite and minimal, then $\mathsf{A}(\G)$ is vigorous by \autoref{lma:DGvigor}, and 
thus $\mathsf{A}(\G)_{x_0} \acts \G^{(0)} \setminus \{x_0\}$ is compressible for all $x_0\in \G^{(0)}$
by \autoref{lma:VigCompress}. Conversely, suppose that 
$\mathsf{A}(\G)_{x_0} \acts \G^{(0)} \setminus \{x_0\}$ is compressible for all $x_0\in \G^{(0)}$. 
Let $X,Y\subseteq \G^{(0)}$ be compact and open subsets. Shrinking $Y$ if necessary, we may assume that $X \cup Y$ is strictly smaller than $\G^{(0)}$. 
Fix $x_0 \in \G^{(0)} \setminus (X \cup Y)$, and use compressibility of 
$\mathsf{A}(\G)_{x_0}\curvearrowright \G^{(0)}\setminus\{x_0\}$ to find a subbasis $\mathcal{U}$ for the topology
of $\G^{(0)}\setminus\{x_0\}$ satisfying the conditions in \autoref{df:ComprAction}. 
By condition (4) in \autoref{df:ComprAction}, there exists $U_1
\in \mathcal{U}$ with $X\subseteq U_1$. Let $U_2 \in \mathcal{U}$ satisfy $U_2 \subseteq Y$, and use (2) in \autoref{df:ComprAction} to find $g \in \mathsf{A}(\G)_{x_0}$ such that $g(X) \subseteq g(U_1) \subseteq U_2 \subseteq Y$. Restricting the source of the bisection corresponding to $g$ to $X$, this shows that 
there is a compact and open bisection $B\subseteq \G^{(0)}$ satisfying $\alpha_B(X)\subseteq Y$. Since $X$ and 
$Y$ are arbitrary, \autoref{lma:PIcompact and open} shows that $\G^{(0)}$ is purely infinite and minimal, as 
desired.
\end{proof}

We now turn to the definition of proper 
characters in a group. We warn the reader that these are not characters in the 
sense of Pontryagin duality,
that is, they are not group homomorphisms to 
the unit circle, and they are only assumed to be
invariant under conjugation. 

\begin{df}
Let $\Gamma$ be a group. A \emph{character}
on $\Gamma$ is a map $\chi\colon \Gamma\to \mathbb{C}$ satisfying the following conditions:
    \begin{itemize}
        \item[(a)] $\chi(gh)=\chi(hg)$ for all $g,h \in \Gamma$;
        \item[(b)] $\chi(1)=1$;
        \item[(c)] for every finite collection $\{g_1,\ldots,g_n\}$ of elements in $\Gamma$, the $n \times n$ matrix with $(i,j)$-th entry $(\chi(g_ig_j^{-1}))$, is non-negatively definite.
    \end{itemize}
    
A character $\chi$ is called \textit{decomposable} if there exist characters $\chi_1,\chi_2$ and a real number $\lambda \in (0,1)$ such that $\chi=\lambda \chi_1 +(1-\lambda) \chi_2$. Otherwise, $\chi$ is called indecomposable. 

The \emph{regular character} is the indecomposable character given by $\chi(g)=0$ whenever $g \neq 1$. The \emph{identity character} is the indecomposable character given by $\chi(g)=1$ for all $g \in \Gamma$. 
We say that $\Gamma$ \emph{has no proper characters} if the only indecomposable characters are the identity character and the regular character.
\end{df}

Using this definition, we deduce the following useful fact about derived subgroups in purely infinite groupoids.

\begin{cor}\label{D has no proper characters}
Let $\G$ be a purely infinite and minimal Cantor groupoid. Then $\mathsf{A}(\G)$ has no proper characters. 
\end{cor}
\begin{proof} 
Note that $\mathsf{A}(\G)$ is vigorous by \autoref{purely infinite if and only if vigor}. 
Fix $x_0\in \G^{(0)}$.
It follows from \autoref{lma:VigCompress} that $\mathsf{A}(\G)_{x_0}\curvearrowright \G^{(0)}\setminus\{x_0\}$ is
compressible.

We claim that $\mathsf{A}(\G)_{x_0}$ is simple. Let $(Z_n)_{n \in \mathbb{N}}$ be an increasing sequence of compact and open subsets of $\G^{(0)} \setminus \Set{x_0}$ such that $\bigcup_{n \in \mathbb{N}} Z_n=\G^{(0)} \setminus \Set{x_0}$. Note that $\G|_{Z_n}$ is a purely infinite minimal Cantor groupoid for all 
$n\in\N$, and thus $\mathsf{A}(\G |_{Z_n})$ is simple by \autoref{minimal essentially principal then simple}.  Therefore $\mathsf{A}(\G)_{x_0}=\bigcup_{n \in \mathbb{N}}\mathsf{A}(\G |_{Z_n})$ is simple as well. 

Applying
\cite[Theorem 2.9]{DudMed_finite_2014}, we deduce that $\mathsf{A}(\G)_{x_0}$ has no proper characters. This implies that $\mathsf{A}(\G)$ has no proper characters: if $\chi$ was a character for $\mathsf{A}(\G)$, then it restricts to a character on $\mathsf{A}(\G)_{x_0}$, therefore $\mathsf{A}(\G)_{x_0} \subseteq \ker(\chi) $, and by simplicity we 
then get $\ker(\chi)=\mathsf{A}(\G)$, as desired.
\end{proof}

We can also obtain some information about the finite factor representations of topological 
full groups. Before we define these representations, 
we recall that for a subset $S \subseteq \mathsf{B}(H)$
of the bounded operators on a Hilbert space $H$, the commutant 
$S'$ is the weak-$\ast$ closed subalgebra of $\mathsf{B}(H)$ given by
\[S'=\{ a\in \mathsf{B}(H)\colon as=sa \mbox{ for all } s\in S\}.\]

\begin{df}
Let $\pi\colon \Gamma \rightarrow \mathsf{B}(H)$ be a unitary representation of a group $\Gamma$ on a Hilbert space $H$. We say that  
$\pi$ is a \emph{finite factor representation} if $\mathcal{M}_\pi=\pi(\Gamma)''$ is a factor, that is, 
$\mathcal{M}_\pi' \cap \mathcal{M}_\pi=\mathbb{C} \text{id}_{H}$.
\end{df}

It is well known that finite factor representations are in one-to-one correspondence with proper characters.
For a groupoid $\G$, we write
\[\pi^\G_{\mathrm{ab}}\colon \mathsf{F}(\G)\to \mathsf{F}(\G)_{\mathrm{ab}}\cong \mathsf{F}(\G)/\mathsf{D}(\G)\] 
for the canonical quotient map (the abelianisation).

\begin{prop}\label{finite factor}
Let $\G$ be a purely infinite, minimal Cantor groupoid, and let 
$\chi\colon \mathsf{F}(\G)\to \mathbb{T}$ be an indecomposable character.
Then $\chi$ is either regular, or has the form $\chi(g)=\rho\circ \pi^\G_{\mathrm{ab}}$ for some 
group homomorphism $\rho\colon \mathsf{F}(\G)_{\mathrm{ab}} \to \mathbb{T}$. In particular, 
the finite factor representations of $\mathsf{F}(\G)$ are all of the form
$g\mapsto \rho(\pi^\G_{\mathrm{ab}}(g))\id_H$, where $\rho$ is a character on $\mathsf{F}(\G)_{\mathrm{ab}}$
and $H$ is a Hilbert space. 
\end{prop}
\begin{proof}
We verify the assumptions of \cite[Theorem 2.11]{DudMed_finite_2014} for $G=\mathsf{F}(\G)$
and $R=\mathsf{A}(\G)$. Note that $\mathsf{A}(\G)$ is ICC by the 
combination of \autoref{purely infinite if and only if vigor} and part~(2) of 
\autoref{lemma:NoInvMeas}; has no proper
characters by \autoref{D has no proper characters}; and is normal in $\mathsf{F}(\G)$.

Let $g \in \mathsf{F}(\G) \setminus \{1\}$. Find a compact and open subset $Y_1\subseteq\G^{(0)}$ such that  $g(Y_1) \cap Y_1= \emptyset$. 
Find an infinite family $\{Y_2^{(n)}\colon n\in\N\}$ of nonempty pairwise disjoint 
compact and open subsets of $Y_1$. 
Since the canonical action of $\mathsf{A}(\G)$ on $\G^{(0)}$ is vigorous by 
\autoref{purely infinite if and only if vigor}, for every $n\in\N$ there exists $h_n\in \mathsf{A}(\G)$ with support contained in $\G^{(0)}\setminus Y_1$, 
such that $h_n(g(Y_1)) \subseteq Y_2^{(n)}$. 
For $n,m\in\N$ distinct, it follows that 
$h_n^{-1}gh_n$ is different from $h_m^{-1}gh_m$.
Moreover, we have
\[ (h_n g^{-1} h_n^{-1}) (h_m g h_m^{-1})=(h_n g^{-1} h_n^{-1}g) (g^{-1}h_m g h_m^{-1})=\big[h_n, g^{-1}\big]\big[g^{-1},h_m\big], \]
which therefore belongs to $\mathsf{D}(\G)=\mathsf{A}(\G)$.
The result now follows from \cite[Theorem 2.11]{DudMed_finite_2014}.
\end{proof}

\subsection{Proof of Main Theorems}
\label{s3last}
In this subsection, we aim to complete the proofs of 
Theorems~\ref{thmintro:equiv} and \ref{thmintro:Realization}.
We begin with Theorem~\ref{thmintro:equiv}, whose statement we reproduce: 

\begin{thm}    \label{thm:equiv}
    Let $\G$ be an essentially principal, ample groupoid. Then the following are equivalent:
    \begin{enumerate}
        \item $\G$ is purely infinite and minimal.
        \item $\mathsf{A}(\G)\leq \mathrm{Homeo}(\G^{(0)})$ is vigorous. 
        \item For every $x_0 \in \G^{(0)}$, the subgroup \[ \mathsf{A}(\G)_{x_0}=\{ g \in \mathsf{A}(\G) \colon \text{there is a neighbourhood } Y \text{ of } x_0 \text{ such that } g|_Y=\id_Y \}\]
        acts compressibly on $\G^{(0)}\setminus \{x_0\}$.  \end{enumerate} 
If any of the above conditions are satisfied, then for every compact open subset $X \subsetneq \G^{(0)}$, there exists an embedding \[\phi_X\colon V \to \mathsf{A}(\G)\] such that $X \subseteq \supp (\phi_X(V))$.
\end{thm}
\begin{proof}
What we have shown in the previous sections proves the result assuming that $\G^{(0)}$ is a Cantor set. 
Indeed, the equivalence between (1) and (2) is \autoref{purely infinite if and only if vigor}; the equivalence between (1) 
and (3) is \autoref{lemma:EquivalencesCompressible}; and the 
fact that
(1) implies the last part of the statement is \autoref{thm:Vembeds}.

We now explain how to obtain the result in full generality from
the Cantor case. Note that the equivalence between (1) and (2) is
automatic,
since being purely infinite and minimal is invariant under restrictions to compact subsets $X \subseteq \G^{(0)}$, and the property in (2) is local. For the same reason, (1) implies the last part of the statement.

It thus remains to show
that (1) is equivalent (3). Assuming (1), let $x_0 \in \G^{(0)}$ be given, and find a compact open cover 
$(X_i)_{i \in I}$ of $\G^{(0)} \setminus \{ x_0 \}$. 
For all $i\in I$, it follows from the compact case that the action \[\mathsf{A}(\G)_{X_i \cup \{x_0\}} \acts X_i \cup \{x_0\} \setminus \{x_0 \}\] is compressible; compressibility of the overall action then follows easily. Conversely, assuming (3), we will show that (1) holds.
Let $X,Y\subseteq \mathcal{G}^{(0)}$ be nonempty 
compact open subsets, and assume without loss of generality that $X\cup Y\neq \mathcal{G}^{(0)}$. Let 
$x_0\notin X\cup Y$, and let $\mathcal{U}$ be a compressible cover associated with the action of $\mathsf{F}(\mathcal{G})\curvearrowright \mathcal{G}^{(0)} \setminus\{ x_0\}$. Let $U_1, U_2 \in \mathcal{U}$ with $X\subseteq U_1$ and $Y\subseteq U_2$ be given. (Such $U_1$ exists because of condition (4) of compressibility.) Using condition (2) of compressibility, find $g \in \mathsf{F}(\G)$ such that $ g(U_1) \subset U_2$. Then $g(X)\subseteq q(U_1)\subseteq U_2 \subseteq Y$, and the restriction $g|_X$ gives a compact open bisection $B$ with $BX \subseteq Y$. 
\end{proof}

We stress the fact that for the groupoids covered by the above theorem, the alternating and derived subgroups always agree by 
\autoref{minimal essentially principal then simple}. 
In particular, the alternating group is always perfect (meaning that it equals its commutator subgroup).

\begin{rem} For Hausdorff groupoids, essential principality is equivalent to effectiveness. On the other hand, effectiveness is the right notion to consider
in the non-Hausdorff setting. We expect the above theorem to hold in this case; however, this involves obtaining (mostly routine) generalisations of many results from the Hausdorff to the non-Hausdorff case. Since for most applications the Hausdorff case is sufficient, we focus on essentially principal groupoids. 
\end{rem}

We now aim to complete our proof of Theorem \ref{thmintro:Realization}. For this, we need to use that the canonical action
$\mathsf{F}(\G) \acts \G^{(0)}$ is rigid in the sense 
of Rubin \cite{Rub_reconstruction_1989}, so we define this notion first. 

\begin{df}
Let $\Gamma$ be a discrete group and let $C$ be a locally 
compact Hausdorff space without isolated points.
We say that an action $\Gamma\acts C$ is a \emph{Rubin action}
if for every open set $U\subseteq C$ and every $x\in U$,
the closure of 
\[\{g\cdot x\colon g\in \Gamma \mbox{ and } \mathrm{supp}(g)\subseteq U\}\]
contains an open neighbourhood of $x$.
\end{df}

Rubin's reconstruction Theorem from \cite{Rub_reconstruction_1989} is extremely useful when working with Rubin actions.

\begin{thm}[Rubin] \label{rubin}
Let $\Gamma_1\acts X_1$ and $\Gamma_2 \acts X_2$ be Rubin actions of discrete groups on compact Hausdorff spaces $X_1$ and $X_2$ without isolated points. If $\rho\colon \Gamma_1 \rightarrow \Gamma_2$ is a group isomorphism, then there exists a 
homeomorphism $\Phi\colon X_1 \rightarrow X_2$ satisfying
$\Phi(g \cdot x)=\rho(g)\cdot \Phi(x)$ for all $x\in X_1$ and all $g\in \Gamma_1$.  
\end{thm}

We will the notion of a \emph{locally
closed} subgroup of homeomorphisms, due to Nyland-Ortega; see \cite[Definition~4.4]{NylOrt_topological_2019}. 

\begin{df}
Let $X$ be a locally compact Hausdorff space and let $\Gamma  \leq \text{Homeo}(X)$ let be a subgroup. We say that a homeomorphism $\phi \in \text{Homeo}(X)$ is \textit{locally in} $\Gamma $ if for all $x \in X$, there exist a neighbourhood $U$ of $x$ and an element $g \in \Gamma $ such that $\phi|_U=g|_U$. We say that $\Gamma $ is \textit{locally closed} if any homeomorphism of $X$ which is locally in $\Gamma $ is automatically in $\Gamma $.
\end{df}

In other words, $\Gamma $ is locally closed if homeomorphisms from it 
can be ``glued" 
to get another homeomorphism of $\Gamma $. It is not difficult to see that topological full groups satisfy this property, 
essentially by construction; see \cite[Theorem~4.5]{NylOrt_topological_2019}. 

\begin{prop}\label{prop:FGgluing}
Let $\G$ be an ample essentially principal groupoid. Then $\mathsf{F}(\G)$ is locally closed.
\end{prop}

Next, we completely characterise
the subgroups of $\mathrm{Homeo}(\mathcal{C})$ that arise as
full groups of minimal, purely infinite groupoids: they are precisely the vigorous groups which are also locally closed. We can also abstractly characterise 
the derived subgroups of such groupoids: they are the 
simple, vigorous groups. 

The following 
is an existence and uniqueness theorem. While the existence part is
completely new, the uniqueness part 
can be deduced from Matui's spatial isomorphism theorem. We nevertheless give a shorter proof using Rubin's Theorem
and \autoref{purely infinite if and only if vigor}, since it illustrates the scope
of our results. 
It should be
pointed out that isomorphism arguments of this nature have also been used in \cite{NylOrt_matuis_2021}.

\begin{thm}\label{thm:Realization}
Let $\mathcal{C}$ denote the Cantor space, and let 
$F\leq \mathrm{Homeo}(\mathcal{C})$ be a subgroup. 
Then the following are equivalent:
\begin{enumerate} 
\item[(F.1)] There exists a minimal, purely infinite, essentially principal, Cantor \'etale groupoid $\G_F$ with $\mathsf{F}(\G_F)\cong F$, and
\item[(F.2)] $F$ is vigorous and locally closed.
\end{enumerate}
For a subgroup $A\leq \mathrm{Homeo}(\mathcal{C})$,
the following are equivalent: 
\begin{enumerate} 
\item[(A.1)] There exists a minimal, purely infinite, essentially principal, Cantor \'etale groupoid $\G_A$ with $\mathsf{A}(\G_A)\cong A$, and
\item[(A.2)] $A$ is vigorous and simple.
\end{enumerate}

Moreover, the groupoids $\G_F$ and $\G_A$ as in (F.1) and (A.1) 
above are unique
up to groupoid isomorphism.
\end{thm}
\begin{proof}
(F.1) implies (F.2). Let $\G$ be a minimal, purely infinite, essentially principal, Cantor \'etale groupoid. Then $\mathsf{F}(\G)$ is locally closed by \autoref{prop:FGgluing}, and it is vigorous by \autoref{purely infinite if and only if vigor}. 

(F.2) implies (F.1). Let $\mathcal{G}$ be the groupoid of germs of the canonical action $F\acts \mathcal{C}$. We claim that $F = \mathsf{F}(\G)$. It is immediate that 
$F\subseteq \mathsf{F}(\G)$, so we only prove the 
converse inclusion. 
Given $g \in \mathsf{F}(\G)$, we use the description of the topological full group of a germ groupoid given in \autoref{TFGdefgroupaction} to
find a 
partition $\mathcal{C}=X_1\sqcup\cdots\sqcup X_n$ of $\mathcal{C}$ into compact and open sets, and group elements $f_1,\ldots,f_n\in F$ such that 
$g|_{X_j}=f_j|_{X_j}$ for all $j=1,\ldots,n$. Since $F$ is locally closed, it follows immediately that $g$ belongs 
to $F$, as desired.

The fact that $\G_F=\mathrm{Germ}(F\acts \mathcal{C})$ is minimal and purely infinite follows from \autoref{purely infinite if and only if vigor}, since its topological full group is 
vigorous.
\vspace{0.2cm}

(A.1) implies (A.2). Similarly to the first part of this proof, vigor of $A$ follows from 
\autoref{purely infinite if and only if vigor}, while 
simplicity follows from 
\autoref{minimal essentially principal then simple}.

(A.2) implies (A.1). Let $F_A\leq \mathrm{Homeo}(\mathcal{C})$ denote the smallest subgroup of 
$\mathrm{Homeo}(\mathcal{C})$ containing $A$ which is locally closed. Since $A$ is vigorous, it is 
immediate to see that so is $F_A$. By the first part of 
this theorem, there is a minimal, purely infinite, essentially principal, Cantor \'etale groupoid $\G_A$
satisfying $\mathsf{F}(\G_A)\cong F_A$. By the equivalence
between (1) and (6) in \cite[Theorem~1.11]{BleEllHyd_sufficient_2020}, it follows that 
$A$ is the commutator subgroup of $F_A$, 
and in combination with \autoref{minimal essentially principal then simple} we get 
$A= [F_A,F_A]\cong \mathsf{D}(\G_A)= \mathsf{A}(\G_A)$, 
as desired. 

\vspace{0.2cm}

We now show the uniqueness part of the statement. We claim that if $\G$ is a purely infinite, minimal Cantor groupoid, then $\mathsf{A}(\G) \acts \G^{(0)}$ is a Rubin action. (This implies that $\mathsf{F}(\G) \acts \G^{(0)}$ is also a Rubin action, since $\mathsf{F}(\G)$ contains $\mathsf{A}(\G)$.) To show this, let $U \subseteq \G^{(0)}$ be a compact and open set and let 
$x,y\in U$. Write $U$ as a disjoint union 
$U=X\sqcup Y$ such that $x \in X$ and $y \in Y$. 
Let $Y_n \subseteq Y$ be a decreasing sequence of compact and open neighbourhoods of $y$ such that $\bigcap_{n \in \mathbb{N}} Y_n=\{y\}$. Since $\mathsf{A}(\G) \acts \G^{(0)}$ is vigorous
by \autoref{lma:DGvigor},
for every $n\in\N$ there exists $g_n \in \mathsf{A}(\G)$ with $\mathrm{supp}(g_n)\subseteq U$ such that $g_n(X) \subseteq Y_n$. Therefore $g_n\cdot x \in Y_n$ for all $n\in\N$, and thus  $y = \lim\limits_{n\to\I} g_n\cdot x$ belongs to the closure of
\[\big\{ g\cdot x \colon g \in \mathsf{A}(\G) \mbox{ and } g|_{\G^{(0)}\setminus U}=\mathrm{id}\big\}.\] 
Since $y\in U$ is arbitrary, this shows that 
\[U=\overline{ \big\{g\cdot x \colon g \in \mathsf{A}(\G) \mbox{ and } g|_{\G^{(0)}\setminus U}=\mathrm{id}\big\} },\]
and hence $\mathsf{A}(\G) \acts \G^{(0)}$ is a Rubin action. This proves the claim.

We now show uniqueness of $\G_F$ in (F.1). Note that, since
$\G_F$ is essentially principal, it is isomorphic to the groupoid of germs
of its canonical action $\mathsf{F}(\G_F)\curvearrowright \G_F^{(0)}$. Now, assume that $\mathcal{H}$ is another minimal, purely infinite, essentially principal, \'etale groupoid with $\mathsf{F}(\mathcal{H})\cong F$. Since $\mathsf{F}(\G_F)$ and $\mathsf{F}(\mathcal{H})$ act in a Rubin manner on the Cantor space by the previous paragraph, \autoref{rubin} 
implies that 
$\mathsf{F}(\G)\curvearrowright \G^{(0)}$ is conjugate to $\mathsf{F}(\mathcal{H})\curvearrowright \mathcal{H}^{(0)}$. In particular, their
groupoids of germs are clearly isomorphic, and thus $\G\cong\mathcal{H}$. 

Finally, we turn to uniqueness of $\G_A$ in (A.1), so let $\mathcal{H}$ be 
a minimal, purely infinite, essentially principal, \'etale groupoid with $\mathsf{A}(\mathcal{H})\cong A$. Since $\mathsf{A}(\G)\curvearrowright \G^{(0)}$ and $\mathsf{A}(\mathcal{H})\curvearrowright \mathcal{H}^{(0)}$ are Rubin actions, it follows from \autoref{rubin} that they are conjugate. Thus, if we
denote by $F_{\mathsf{A}(\mathcal{H})}$ the smallest subgroup of $\mathrm{Homeo}(\mathcal{H}^{(0)})$ which is locally closed and contains $\mathsf{A}(\mathcal{H})$,
it follows that $F_A\cong F_{\mathsf{A}(\mathcal{H})}$. On the other hand, it is 
routine to check that $F_{\mathsf{A}(\mathcal{H})}$ is precisely  $\mathsf{F}(\mathcal{H})$. Putting these things together, we conclude that 
\[\mathsf{F}(\G_A)\cong F_A\cong F_{\mathsf{A}(\mathcal{H})}\cong \mathsf{F}(\mathcal{H}).\]
Thus $\mathcal{H}\cong \G_A$ by uniqueness in (F.1).
\end{proof}

\begin{rem}
The proof above also shows that 
if $A\leq F\leq\mathrm{Homeo}(\mathcal{C})$ are 
nested groups such that $A$ is vigorous and simple
and $F$ is locally closed, then there exists a
unique minimal, purely infinite, topologically principal,
\'etale Cantor groupoid $\G$ such that $A\cong\mathsf{A}(\G)$ and $F\cong \mathsf{F}(\G)$. In particular,
normality of $A$ in $F$ follows from the remaining
assumptions. 
\end{rem}

As a corollary, we can deduce that simple, vigorous
subgroups of $\mathrm{Homeo}(\mathcal{C})$ are 
C*-simple. More precisely:

\begin{cor}\label{cor:C*simple}
Let $A\leq F\leq\mathrm{Homeo}(\mathcal{C})$ be 
nested groups such that $A$ is vigorous and simple,
and $F$ is locally closed. Then any intermediate
group $A\leq H\leq F$ is C*-simple. 
\end{cor}
\begin{proof}
By \autoref{thm:Realization} and the remark after it, 
there exists a minimal, purely infinite, 
topologically principal, \'etale Cantor 
groupoid $\G$ such that $A\cong\mathsf{A}(\G)$ and $F\cong \mathsf{F}(\G)$. Using vigor of
$\mathsf{A}(\G)\acts \G^{(0)}$, and in particular the 
absence of invariant probability measures (\autoref{lemma:NoInvMeas}), the arguments in \cite{scarparo2021dichotomy} can be adapted in a routine 
way to show that both $\mathsf{A}(\G)$ and 
$\mathsf{F}(\G)$ are C*-simple. Instead of doing this,
we give a more direct proof using the recent
results in \cite{B_dos_2023}; we thank Eduardo Scarparo
for sharing this argument with us and for allowing us
to include it here.

By \autoref{cor:RelativeICC}, the $\mathsf{D}(\G)$-conjugacy classes of $\mathsf{F}(\G)$ are infinite. Hence, it follows from \cite[Theorem 6.2]{B_dos_2023}, that every intermediate C$^*$-algebra 
$C^*_r(\mathsf{A}(\G)) \leq A  \leq C^*_r(\mathsf{F}(\G))$ is simple. Since $C^*_r(H)$ is of that form,
the result follows.
\end{proof}

Here is another direct application of \autoref{thm:Realization}.

\begin{cor}
    Vigorous simple groups have no proper characters. 
\end{cor}

One can combine \autoref{thm:Realization} with \autoref{thm:equiv} in order to obtain further equivalent conditions for a group to be 
realized as the alternating group of a purely infinite, minimal groupoid. 
More specifically:

\begin{rem}
    Let $\mathcal{C}$ denote the Cantor space and let $A \leq \mathrm{Homeo}(\mathcal{C})$ be a subgroup. Then the following are equivalent:
    \begin{enumerate} 
\item[(A.1)] There exists a minimal, purely infinite, essentially principal, Cantor \'etale groupoid $\G_A$ with $\mathsf{A}(\G_A)\cong A$, and
\item[(A.2)] $A$ is vigorous and simple.
 \item[(A.3)] $A$ is simple and for all $x \in \mathcal{C}$, the subgroup 
 \[A_x=\{g \in A \colon \text{there exists a neighbourhood } Y \text{ of } x \text{ such that } g|_Y=\id_Y\}\] 
 acts compressibly on $\mathcal{C} \setminus \{x\}$.
\end{enumerate}
   \end{rem}

We close this section by pointing out that \autoref{thm:Realization}
can be generalized to non-compact unit spaces. Indeed, the only step in our arguments where compactness is needed is in the uniqueness part,
since we use Rubin's theorem which is valid only for compact spaces.
The way to deal with non-compact spaces is to use the generalisation of Matui's spatial isomorphism proved by Nyland-Ortega in \cite[Theorem~7.2]{NylOrt_topological_2019} in the setting where we replace $\mathcal{C}$ with an arbitrary totally disconnected locally compact Hausdorff space $X$. In this context, one must also replace $\text{Homeo}(\mathcal{C})$ with the group of compactly supported homeomorphisms:
\[\text{Homeo}_\mathrm{c}(X)=\{f \in \text{Homeo}(X)\colon \text{supp}(f) \text{ is compact}\}.\] 
We omit the details, but stress the fact that these arguments also
allow one to deal with non-Hausdorff groupoids since they only depend
on effectiveness. 

\section{Examples and applications}
\label{s5}

For the purposes of this section, we use groupoid homology, primarily as a way to compute the quotients $\mathsf{F}(\G)/\mathsf{D}(\G)$ for key examples. Fortunately, groupoid homology has already been computed in many interesting examples. For more on groupoid homology, we refer the reader to \cite{Mat_homology_2012}. Thanks to Li's proof of the AH-conjecture for a large class of ample groupoids \cite[Corollary E]{Li_ample_2022}, we now have very concrete pictures of the abelianisation $\mathsf{F}(\G)_{\mathrm{ab}}$ of the full groups of these groupoids in terms of their groupoid homology:

\begin{thm}\label{ah conjecture}
(\cite[Corollary~E]{Li_ample_2022}).
Let $\G$ be an ample, minimal groupoid with comparison\footnote{Comparison is a notion for groupoids defined in the second paragraph of page 6 of \cite{Li_ample_2022}. For our purposes, it suffices to mention that comparison is automatic for purely infinite, minimal 
groupoids; see the comments at the top of page 5 of \cite{Li_ample_2022}. } and such that
$\G^{(0)}$. Then there exists a short exact sequence:
\[\xymatrix{H_2(\G) \ar[r]& H_{0}(\G )\otimes \mathbb{Z}_2 \ar[r]^-{j} &\mathsf{F}(\G)_{\mathrm{ab}} \ar[r]^-I& H_1(\G)
\ar[r] &0.}\] \end{thm}

Our first application giving us restrictions on what kinds of actions 
are possible for $\mathsf{F}(\G)$.

\begin{prop} \label{essentially free automatic} 
Let $\G$ be a purely infinite minimal groupoid such that $\mathsf{F}(\G)_{\mathrm{ab}}$ is finite (for example, if $H_1(\G)$ is finite and $H_0(\G)$ has finite rank). Then every faithful ergodic measure-preserving action of $\mathsf{F}(\G)$ is essentially free. 
\end{prop}
\begin{proof}
We saw in \autoref{finite factor} that every finite factor representation factors through the abelianisation $\mathsf{F}(\G)_{\mathrm{ab}}$. Therefore, whenever $\mathsf{F}(\G)_{\mathrm{ab}}$ is finite, there are finitely many factor representations. We may then apply \cite[Theorem 2.11]{DudMed_finite_2014} to get the conclusion.
\end{proof}

We can be even more concrete in some cases where homology has been computed. 

\begin{cor}
For each of the following groups, every faithful ergodic 
measure-preserving action is essentially free:
\begin{enumerate}
\item The Brin-Higman-Thompson groups $nV_{k,r}$. 
\item More generally, topological full groups arising from products of shift of finite type groupoids.
\end{enumerate}
\end{cor}
\begin{proof} It is shown in \cite{Mat_etale_2016} that the abelianisations
of the topological full groups of such groupoids are finite. Therefore the result follows from \autoref{essentially free automatic}.\end{proof}

As another application, we confirm 
Conjecture~\ref{fpresent2present} for many natural examples including the class of groups above. 

\begin{eg}
The following groups, which are already known to be finitely presented and simple, are 2-generated. 
   \begin{enumerate}
   \item Derived subgroups of the topological full groups of Graph groupoids . 
   \item Derived subgroups of the topological full groups of Katsura-Exel-Pardo groupoids.
   \item Simple subgroups arising as subgroups of topological full groups arising from products of shifts of finite type. 
   \item Derived subgroups of the topological full groups of Beta expansion groupoids. 
       \item Certain simple groups arising from groupoids that are left regular representations of Garside categories.  
   \end{enumerate} 
   In particular, Conjecture~\ref{fpresent2present} holds true for these examples.
\end{eg}
\begin{proof}
The groups listed above are known to be the topological full groups of purely infinite, minimal, ample, and effective groupoids:
for (1), this is shown in \cite{NylOrt_topological_2019}; 
for (2), this is shown in \cite{NylOrt_katsura_2021}; 
for (3), this is shown in \cite{Mat_etale_2016};
for (4), this is shown in \cite{MatMat_topological_2014}; 
and for (5) this is shown in \cite[Theorem~C]{li2021left} and 
\cite[Theorem~A]{li2022cstar}. 
By Theorem~\ref{thmintro:equiv}, these groups act vigorously on the Cantor set. It thus follows from \cite[Theorem~1.12]{BleEllHyd_sufficient_2020} that these groups are all 2-generated.  
\end{proof}

Next, we show C*-simplicity for many examples of Thompson-like groups. 

\begin{eg}
The following Thompson-like groups are C*-simple: 
    \begin{enumerate}
    \item The Higman-Thompson groups $V_{k,r}$, for $k\geq 2\in\N$
    and $r\geq 1$.
        \item More generally, 
        the Brin-Higman-Thompson groups $nV_{k,r}$, for $r,n\in\N$ and $k\geq 2$.
        \item Irrational slope Thompson groups.
    \end{enumerate}
\end{eg}
\begin{proof}
All these groups act vigorously on the Cantor set and are locally closed; in other words, they are topological full groups of minimal, purely infinite Cantor groupoids
by \autoref{thm:Realization}. The result then follows from 
\autoref{cor:C*simple}. 
\end{proof}

Using Theorem~\ref{thmintro:Realization}, we can bring new groups into our framework. As an example, we discuss the twisted Brin-Higman-Thompson groups introduced by Belk and Zaremsky \cite{belk2022twisted}, which were previously not known to be topological full groups of 
purely infinite, minimal groupoids. 
We thank James Belk for bringing this example class to our attention and for an outline of this proof.

\begin{eg}[Twisted Brin-Higman-Thompson groups]
The twisted Brin-Higman-Thompson groups $SV$ were introduced in \cite{belk2022twisted}. Already in their construction, we see that they are topological full groups, and hence they are locally closed. In the case where $S$ is finite, say $|S|=n$, it follows that $nV$ embeds into $SV$ with full support, and therefore $SV$ is vigorous. 
Similarly, if $S$ is countable, we have that $nV$ embeds into $SV$ for all $n$ in a way that $\bigcup_{n\in\N} \mathrm{supp}(nV)$ covers the whole of the Cantor space. Therefore, $SV$ is again vigorous. We conclude that $SV$ is always a locally closed, vigorous subgroup of $\mathrm{Homeo}(\mathcal{C})$. Applying Theorem~\ref{thmintro:Realization}, we deduce that $SV$ is the topological full group of an essentially principal, purely infinite, minimal Cantor \'etale groupoid. 
\end{eg}

As a further application, we can describe all proper characters and finite factor representations of the Brin-Higman-Thompson groups
$nV_{k,r}$. Our discussion parallels and extends that of Matui in \cite[Section~5.5]{Mat_etale_2016}, and we take this opportunity to rectify a minor typo
in \cite[Theorem~5.20]{Mat_etale_2016}: in the case where $n=1$, we have that $ (V_{k,r})_{\mathrm{ab}}$ vanishes when $k$ is even and is $\mathbb{Z}_2$ for $k$ odd, and not the other way around.

\begin{eg}
\label{brin}
Let $n,k,r\in\N$, and consider the groupoid $\mathcal{R}_r \times \mathcal{E}_k^n$ consisting of the $r$-stabilisation of the $n$-fold product of the shift of finite type groupoid $\mathcal{E}_k$ associated to the single vertex graph with $k$ loops. Then $\mathsf{F}(\mathcal{R}_r \times \mathcal{E}_k^n)$ has been explicitly identified with the Brin-Higman-Thompson group $nV_{k,r}$ in \cite{Mat_etale_2016}. By \cite[Theorem~F]{Li_ample_2022}, 
groupoid homology is independent of $r$, and 
it is computable via the K\"unneth formula, yielding: 
\[H_i(\mathcal{R}_r \times \mathcal{E}_k^n)\cong \mathbb{Z}_{k-1}^{^{n-1}C_i}\] for all $i\geq 0$,
where $^nC_i$ is the binomial coefficient $^nC_i=\tfrac{n!}{(n-i)!i!}$. 
Plugging the homology of the groupoid for the Brin-Higman-Thompson group $nV_{k,r}$ into \autoref{ah conjecture}, we obtain: 
\[\xymatrix{
\mathbb{Z}_{k-1}^{^{n-1}C_2}\ar[r]^-{h} & \mathbb{Z}_{k-1} \otimes \mathbb{Z}_2 \ar[r]^-{j}& (nV_{k,r})_{\mathrm{ab}}  \ar@{-->}@/^1.0pc/@[black][l]^{\rho} \ar[r]^-{I}& \mathbb{Z}_{k-1}^{{n-1}} \ar[r]& 0.
}\]
We divide the discussion into cases:
\begin{enumerate}
    \item If $k$ is even, then $\mathbb{Z}_{k-1} \otimes \mathbb{Z}_2=0$ and hence $I$ is an isomorphism. In this case, we get $(nV_{k,r})_{\mathrm{ab}}=\mathbb{Z}_{k-1}^{n-1}$. 
    \item If $k \in 4 \mathbb{Z} + 1$, say $k=4m+1$ we can see that the map 
    \[j\colon \mathbb{Z}_2\cong \mathbb{Z}_{4 m} \otimes \mathbb{Z}_2 \to (nV_{4m+1, r})_{\mathrm{ab}}\] is injective. Thus there is a homomorphism $\rho\colon (nV_{k,r})_{\mathrm{ab}} \to \mathbb{Z}_2$ such that $j \circ \rho= \mathrm{id}_{(nV_{k,r})_{\mathrm{ab}}}$, so the short exact sequence splits. In this case we get
    $(nV_{k,r})_{\mathrm{ab}}\cong \mathbb{Z}_2\oplus \mathbb{Z}_{k-1}^{n-1}$.
    \item If $k \in 4 \mathbb{Z} + 3$, say $k=4m+3$, then we need to consider the cases $n=1, n=2$ and $n>2$ separately. When 
        \begin{enumerate}
            \item $n=1$: then $j$ is necessarily injective on $\mathbb{Z}_{k-1} \otimes \mathbb{Z}_2$ since $H_2(\mathcal{R}_r \times \mathcal{E}_k)=0$. Moreover, $j$ is an isomorphism since $H_1(\G)$ also vanishes. 
            \item $n=2$: again $j$ is necessarily injective on $\mathbb{Z}_{k-1} \otimes \mathbb{Z}_2$ since $H_2(\mathcal{R}_r \times \mathcal{E}_k)=0$. Additionally, one can show as in the proof of \cite[Lemma 5.19(3)]{Mat_etale_2016} that there exists a surjection $2V_{k,r}\rightarrow \mathbb{Z}_{2m-2}$, so that the abelianisation must be $\mathbb{Z}_{2m-2}$.
        \item $n>2$: then $j=0$ because $h\colon \mathbb{Z}_{4 m+1}^{^{n-1}C_2} \to \mathbb{Z}_2$ is a surjection. 
    \end{enumerate}
\end{enumerate}

In summary, for arbitrary $n,r,k\in\N$ with $k \geq 2$ and $r,n \geq 1$, the abelianisation of the Brin-Higman-Thompson groups is given by:
$$(nV_{k,r})_{\mathrm{ab}}=\begin{cases}
0 & k=2,  \text{ and } n>1; \\
0 & k>2 \text{ is even and }n=1;\\
\mathbb{Z}_{k-1}^{n-1} &  k \text{ even and } n>1 \\
\mathbb{Z}_2 & k \in  2\mathbb{Z}+1  \text{ and } n=1; \\
\mathbb{Z}_2 \oplus \mathbb{Z}_{k-1}^{n-1} & k \in  4\mathbb{Z}+1  \text{ and } n>1; \ \\
\mathbb{Z}_{2k-2}  & k \in 4\mathbb{Z}+3  \text{ and } n=2;\\
\mathbb{Z}_{k-1}^{n-1} &  k \in 4\mathbb{Z}+3  \text{ and } n>2. \end{cases}
$$
Therefore, using \autoref{finite factor} we can concretely describe all the proper characters, which are in one-to-one correspondence with finite factor representations of $nV_{k,r}$. The outcome is the following:
\begin{itemize}
    \item[(i)] $nV_{2,r}$ has no proper characters, and $V_{k,r}$ has no proper characters if $k$ is even. 
    \item[(ii)] If $k>2$ is even and $n>1$, then $nV_{k,r}$ has $n-1$ characters, all of order $k-1$. Also, if $k \in 4\mathbb{Z}+3$ and $n>2$, then $nV_{k,r}$ has $n-1$ characters, all of order $k-1$.
    \item [(iii)] If $k$ is odd and $n=1$, then $V_{k,r}$ has one proper character of order 2.
    \item[(iv)] If $k \in 4\mathbb{Z}+1$ and $n>1$, then $nV_{k,r}$ has $n$ proper characters, one of order $2$ and the others of order $k-1$. 
    \item[(v)] If $k \in 4 \mathbb{Z}+1$ and $n=2$, then $2V_{k,r}$ has one proper character of order $2k-2$. 
\end{itemize}
\end{eg}

We end with an example that generalises the above, advertising the flexibility that comes with working with topological full groups of \'etale groupoids. These groups
have recently turned out to be very relevant in the study of embeddings of 
tensor products of $L^p$-Cuntz algebras; see \cite{GarGun_embeddings_2023}.

\begin{eg}[Perfect Brin-Higman-Thompson like groups]
Consider the variation of the Brin-Higman-Thompson groups on $r^n$ cubes such that on each dimension $j=1,\ldots,n$ we have potentially different (integrally generated) slope sets generated by integers $k_j\geq 2$. 
Write $\overline{k}$ for the $n$-tuple $\overline{k}=(k_1,\ldots,k_n)$,
and denote the group described above by $V_{\overline{k},r}$. Note 
that, if $k_1=\cdots=k_n=:k$, then $V_{\overline{k},r}$ is just the Brin-Higman-Thompson group $nV_{k,r}$.

It is not difficult to see that $V_{\overline{k},r}$ is the topological full group of the groupoid:
\[\mathcal{G}_{\overline{k},r}:=\mathcal{R}_r \times \prod_{j=1}^n \mathcal{E}_{k_j}.\]
This subclass of products of shifts of finite type groupoids fits into the framework of \cite{Mat_etale_2016}, and it is not difficult to see that if we set $g=\mathrm{gcd}(k_1-1,\ldots,k_n-1)$, then for $j\geq 0$ we have:
\[H_j\big(\mathcal{G}_{\overline{k},r}\big)=(\mathbb{Z}/g \mathbb{Z})^{^{(n-1)}C_j}.\]
    \end{eg}

In particular, if in the discussion above we have $g=1$, for example if $k_j=2$ for some $j$, then the homology vanishes. 

\begin{prop}
Let $k_1,\ldots,k_n\geq 2$ satisfy $\mathrm{gcd}(k_1-1,\ldots ,k_n-1)=1$. 
Then $V_{\overline{k},r}$ has the following properties: 
\begin{enumerate}
\item It is acyclic, so in particular perfect and simple.
         \item It is $F_\infty$; in particular it is finitely presented.
         \item It is 2-generated and C*-simple. 
         \item It has no proper characters. 
         \item Every faithful ergodic measure preserving action of $V_{\overline{k},r}$ is essentially free. 
     \end{enumerate}
\end{prop}
\begin{proof}
We observed above that $\mathrm{gcd}(k_1-1,\ldots,k_n-1)=1$ implies that the homology of $\mathcal{G}_{\overline{k},r}$ vanishes. Inserting this computation into \autoref{ah conjecture}, it is easily seen that 
$V_{\overline{k},r}=\mathsf{F}(\mathcal{G}_{\overline{k},r})$
is perfect in this case. Moreover, by \cite[Corollary~D]{Li_ample_2022}, it follows that $V_{\overline{k},r}$ is acyclic. The rest of the claims
follow from our results.
\end{proof}


\section{Outlook}

Our first two questions address obvious ‘gaps’ in Table~\ref{table 1}. In \cite{brin2007algebra}, the so-called braided Higman-Thompson groups were introduced; these groups exhibit many shared properties with $V$ and are often type $F_\infty$. It would be interesting to understand if these groups fit into our framework.   

\begin{qst}\label{qst1}
Can the braided Higman-Thompson groups be realised as topological full groups of some purely infinite minimal groupoids?
\end{qst}

Another class which has recently been of interest is certain full automorphism groups $V_r(\Sigma)$ of Cantor algebras, as described in \cite{cohomological}, which have been shown to be type $F_\infty$ whenever the underlying Cantor algebras $U_r(\Sigma)$ are \textit{valid, bounded} and \textit{complete}. In a similar spirit to the question above, we ask:

\begin{qst}\label{qst2}
 Can the family of groups $V_r(\Sigma)$, for $U_r(\Sigma)$ valid, bounded and complete, be described as the topological full groups of purely infinite minimal groupoids?
\end{qst}

In view of Theorem \ref{thmintro:Realization}, Questions~\ref{qst1} and~\ref{qst2} are equivalent to asking whether the (simple) derived subgroups of the groups described admit vigorous actions on the Cantor space. The difficulty in both cases above is that the groups in
question are not constructed as subgroups of homeomorphisms of the
Cantor set.

The next question relates to Conjecture~\ref{fpresent2present} in the introduction. Topological full groups provide interesting examples of simple, finitely generated groups outside of the purely infinite setting, but it is currently not known if any example outside of the purely infinite setting is finitely generated.
We therefore ask:

\begin{qst}
Are there finitely presented, minimal topological full groups outside of the purely infinite setting? \end{qst}

Our final question addresses the scope of Corollary~\ref{nopropercharacters}. In recent work in preparation by Dudko-Medynets, they show that for AF-groupoids, there is a one-to-one correspondence between proper characters of the topological full group and invariant probability measures on the unit space. Corollary~\ref{nopropercharacters} shows that the same is true for purely infinite minimal groupoids (since these do not admit invariant probability measures). It seems reasonable to believe
that this correspondence may exist in full generality, so we ask:

\begin{qst} Let $\G$ be an \'etale, essentially principal Cantor groupoid. Is there a one-to-one correspondence between invariant probability measures on $\G$ and a proper characters on $\mathsf{A}(\G)$? \end{qst}

\bibliographystyle{plain}


\end{document}